\documentclass[11pt]{article}
\usepackage[T1]{fontenc}
\usepackage{lmodern}
\usepackage[a4paper,margin=1in]{geometry}
\usepackage{amsmath,amssymb,amsfonts,amsthm,mathtools}
\usepackage{booktabs,multirow,array}
\usepackage{tabularx}
\usepackage{graphicx,float}
\usepackage{placeins}
\usepackage{enumitem}
\usepackage{microtype}
\usepackage{xcolor}
\usepackage{hyperref}
\hypersetup{
	colorlinks=true,
	linkcolor=blue!55!black,
	citecolor=blue!55!black,
	urlcolor=blue!55!black,
	pdftitle={A Positivity-Preserving Expectation Scheme for HJB Equations with Oblique Robin Boundary Conditions},
	pdfauthor={Haoran Xu and Xingye Yue},
	pdfkeywords={Hamilton-Jacobi-Bellman equation, oblique Robin boundary condition, conditional expectation, reflected diffusion, positivity-preserving scheme, viscosity solution}
}
\setlist{leftmargin=2em}
\allowdisplaybreaks

\newtheorem{theorem}{Theorem}[section]
\newtheorem{proposition}[theorem]{Proposition}
\newtheorem{lemma}[theorem]{Lemma}

\theoremstyle{definition}
\newtheorem{definition}[theorem]{Definition}
\theoremstyle{remark}
\newtheorem{remark}[theorem]{Remark}

\newcommand{\R}{\mathbb{R}}
\newcommand{\E}{\mathbb{E}}
\newcommand{\cO}{\mathcal{O}}
\newcommand{\cG}{\mathcal{G}}
\newcommand{\cS}{\mathcal{S}}
\newcommand{\Tr}{\operatorname{Tr}}

\newcommand{\dist}{\operatorname{dist}}
\newcommand{\dt}{\Delta t}

\title{A Positivity-Preserving Expectation Scheme for Hamilton--Jacobi--Bellman Equations with Oblique Robin Boundary Conditions}
\author{Haoran Xu$^1$ \and Xingye Yue$^{2,*}$}
\date{}

\begin{document}
	\maketitle
	\begin{center}
		\small
		$^1$School of Mathematical Sciences, Soochow University, Suzhou 215006, China.\\
		$^2$Center for Financial Engineering, School of Mathematical Sciences, Soochow University, Suzhou 215006, China.\\
		$^*$Corresponding author: \texttt{xyyue@suda.edu.cn};
		Haoran Xu: \texttt{20244007005@stu.suda.edu.cn}.
	\end{center}
\begin{abstract}
	For anisotropic diffusion with mixed derivatives, standard compact
	coordinate-aligned stencils and some local finite-volume or finite-element
	constructions can lose nonnegative coefficients unless suitable coefficient
	or mesh conditions are imposed.  Here nonnegative coefficients are generated
	directly from conditional expectation rather than from an algebraic stencil
	decomposition. We construct a positivity-preserving scheme for
	possibly degenerate Hamilton--Jacobi--Bellman equations with controlled
	oblique Robin boundary conditions. At interior nodes the construction stems
	from a one-step reflected Feynman--Kac identity: an $m$-dimensional Rademacher
	vector generates $P=2^m$ equally probable weak-Euler branches, an exterior
	branch is mirrored about its oblique projection, and twice its overshoot
	serves as the discrete boundary local time $D$. The Robin coefficient enters
	only through the attenuation factors $e^{-\kappa D}$ and $De^{-\kappa D/2}g$,
	leaving the equal branch probabilities unchanged. Boundary nodes use a
	separate same-level closure over a spatial offset $\ell_h\asymp h$, which may
	be implicit; uniform obliqueness and $\mathbb P_1$ interpolation give a
	mesh-independent positive weight on interior nodes, so fixed linear Robin data
	yield a sparse nonsingular $M$-matrix system and the general control set
	yields a monotone contraction. Positivity is preserved for nonnegative data
	without a diagonal-dominance condition and without any CFL-type relation
	between $\Delta t$ and $h$.

Under $\dt,h\to0$ with $h^2/\dt\to0$ we prove consistency, uniform
stability, and convergence to the viscosity solution via half-relaxed limits.
For a classical solution $u\in C_b^3(\mathcal N_T)$ whose control-resolved
boundary residuals $q_b:=\mathcal L(t,\cdot,Du,u,b)$ are uniformly $C^{1,1}$
on $\partial\cO$, a three-term error decomposition and backward recursion give
$\max|U_i^n-u(t_n,x_i)|\le C(\dt^{1/2}+h^2/\dt)$.  The two terms are balanced
by $\dt\asymp h^{4/3}$, yielding $O(h^{2/3})$, whereas $\dt\asymp h$ gives
$O(h^{1/2})$.  Four two-dimensional studies show monotone error decay over
the reported refinements.  An angular-control sensitivity study is consistent
with the $O(K^{-2})$ directional defect.  In the smooth variable-Robin test,
the balanced schedule $\dt\propto h_{\rm real}^{4/3}$ gives final observed
orders $1.242$--$1.261$ over the displayed refinements, sharper than the
conservative $O(h^{2/3})$ bound.  Nonsmooth conical and mixed
Robin--Dirichlet examples are reported only as stress tests.

\end{abstract}

	\noindent\textbf{Keywords.} Hamilton--Jacobi--Bellman equation; oblique Robin boundary condition; conditional expectation; reflected diffusion; boundary local time; positivity preservation; viscosity solution; error estimate.
	
	\medskip
	\noindent\textbf{2020 Mathematics Subject Classification.} 65M12, 65M15, 49L25, 35D40.
	
\section{Introduction}

For anisotropic or degenerate diffusion, classical finite-difference,
finite-volume, and finite-element discretizations cannot in general be both
consistent and positivity-preserving unless the diffusion matrix is
diagonally dominant~\cite{MotzkinWasow1952,CrandallLions1996,NordbottenEtAl2007,SharmaHammett2007}.
Nonlinear compact corrections~\cite{GaoWu2015,KuzminShashkovSvyatskiy2009,LiEtAl2022,SharmaHammett2007,ZhangSuWu2017,ShengYuan2008,YuanSheng2007,YuanSheng2008}
and wide-stencil directional splittings~\cite{BonnansZidani2003,DebrabantJakobsen2013}
mitigate this at the price of complexity or boundary extrapolation, and
model-specific constructions exist~\cite{MaForsyth2017,XuChenLiuYue2019}.

Positivity is closely related to monotonicity, the comparison principle, and
the maximum principle.  For linear problems, positivity and monotonicity are
equivalent; for nonlinear problems, positivity is necessary but not
sufficient for monotonicity.  In discrete schemes, monotonicity generally
ensures $L^\infty$-stability, which is essential for reliable computations.
For nonlinear problems, the viscosity-solution framework provides the
theoretical foundation: consistent and monotone schemes converge to the
viscosity solution of the underlying PDE~\cite{BarlesSouganidis1991}.

By the Feynman--Kac formula the solution is a conditional expectation of a
diffusion process, and discretizing the expectation yields nonnegative
weights. Classical tree~\cite{Duffy2006}, Markov-chain~\cite{KushnerDupuis2001}, and
Monte~Carlo~\cite{FahimTouziWarin2011,GuoZhangZhuo2015} approaches face
limitations at boundaries or in low dimensions.  We recently built a
non-compact positivity-preserving framework based on conditional
expectation~\cite{RenXuYue2024,XuRenYue2026,XuLiYue2026}, which, however, does
not treat inhomogeneous Robin boundary conditions.  Robin data arise in
interfacial phenomena such as heat convection and electrochemistry; for HJB
equations, monotone, consistent, and stable schemes converge to the viscosity
solution by the Barles--Souganidis framework~\cite{BarlesSouganidis1991}.
Monte~Carlo methods for reflected diffusions~\cite{LeimkuhlerEtAl2023,NystromOnskog2010,ZhouCai2016}
are positivity-preserving but are not deterministic grid schemes.  Calzola,
Carlini, Dupuis, and Silva~\cite{CalzolaEtAl2023} proposed a convergent
semi-Lagrangian scheme for oblique derivative conditions, corresponding to
$k=0$ here, using an inward offset with a tunable parameter $\bar c$.

This paper extends our expectation framework to HJB equations with oblique
Robin boundary conditions, adding the zero-order term $ku$ to the boundary
operator of~\cite{CalzolaEtAl2023}.  The scheme is built from the reflected
Feynman--Kac representation: at each interior node, $P=2^m$ equally probable
Rademacher branches are generated; an exterior branch is mirrored about its
oblique projection, $\widetilde X=X^*-2d^{\gamma_b}(X^*)\gamma_b(p^{\gamma_b}(X^*))$,
and the round-trip distance $\widetilde d=2d^{\gamma_b}(X^*)$ is the discrete
boundary-local-time increment.  Unlike~\cite{CalzolaEtAl2023}, the reflection
is fully geometric and requires no additional inward-shift coefficient
$\bar c$; the Robin coefficient enters only
through the attenuation factors $e^{-\widetilde k\widetilde d}$ and
$\widetilde d\,e^{-\widetilde k\widetilde d/2}\widetilde g$, leaving the
probability $1/P$ unchanged.  At boundary nodes the Robin equation is
discretized separately by a same-level $\mathbb P_1$ closure over a spatial
offset $\ell_h\asymp h$, which may be implicit: fixed linear data give a
sparse nonsingular $M$-matrix system and the controlled boundary map is a
contraction.  For nonnegative data the scheme preserves nonnegativity without
a diagonal-dominance condition and without any CFL-type relation between
$\dt$ and $h$.

Under $\dt,h\to0$ with $h^2/\dt\to0$ we prove separate interior- and
boundary-node consistency, uniform $L^\infty$ stability, and convergence to
the viscosity solution.  For classical solutions with uniformly $C^{1,1}$
Robin residuals we obtain
$\max|U_i^n-u(t_n,x_i)|\le C(\dt^{1/2}+h^2/\dt)$.  This bound gives
$O(h^{2/3})$ under the balancing choice $\dt\asymp h^{4/3}$ and
$O(h^{1/2})$ under the commonly used choice $\dt\asymp h$.

The main contributions are summarized as follows.
\begin{enumerate}[label=(\roman*)]
	\item \textbf{A positive expectation discretization for anisotropic and
	degenerate diffusion.}  The diffusion
	operator is replaced by an expectation of directional derivatives: the
	discrete weights are products of the equal probability $1/P$, nonnegative
	$\mathbb P_1$ weights, and nonnegative attenuation factors, so anisotropic
	and degenerate diffusion is handled without imposing a compact-stencil
	diagonal-dominance condition.  This is the volume mechanism on which the new
	Robin construction below is built.
	\item \textbf{Systematic Robin treatment by mirror reflection and boundary
		local time.}  Unlike~\cite{CalzolaEtAl2023}, the scheme uses mirror reflection
	without the additional coefficient $\bar c$; the round-trip distance
	$\widetilde d$ serves as
	the discrete local-time increment and couples with $k$ through
	$e^{-\widetilde k\widetilde d}$ and
	$\widetilde d\,e^{-\widetilde k\widetilde d/2}\widetilde g$.  The Robin
	condition is discretised separately at boundary nodes by a same-level
	closure: a sparse $M$-matrix system for fixed linear data and a monotone
	contraction in the controlled case.
	\item \textbf{Monotone and $L^\infty$-stable with a full convergence theory.}
	The interior update is explicit, and positivity and monotonicity impose no
	time--space coupling.  Under $r,k\ge0$ and the stated refinement condition,
	the barrier argument gives a uniform $L^\infty$ bound.  Monotonicity,
	consistency, and stability yield, in the
	Barles--Souganidis sense~\cite{BarlesSouganidis1991}, convergence to the unique viscosity
	solution under $\dt,h\to0$, $h^2/\dt\to0$.
	\item \textbf{Quantitative error estimate and numerical verification.}  For
	classical solutions $u\in C_b^3(\mathcal N_T)$ whose control-resolved boundary
	residuals $q_b:=\mathcal L(t,x,Du,u,b)$ are uniformly $C^{1,1}$ on
	$\partial\cO$, we obtain the full-grid estimate $O(\dt^{1/2}+h^2/\dt)$, i.e.\
	$O(h^{2/3})$ for the balancing choice $\dt\asymp h^{4/3}$ and
	$O(h^{1/2})$ when $\dt\asymp h$.  Four two-dimensional experiments
	(a normal/oblique $k=0$ comparison, a nonzero-Robin test, a nonsmooth
	conical test, and a mixed Robin--Dirichlet test) show monotone error decay
	over the reported refinements.  The smooth Robin test includes an angular
	control refinement and a direct comparison of $\dt\asymp h$ with the
	balancing scale $\dt\asymp h^{4/3}$; the latter gives final observed orders
	$1.242$--$1.261$ over the displayed refinements, while the angular
	differences decay consistently with $O(K^{-2})$.
	For the $k=0$ comparison, the reference offset is selected by a stated
	coarse-grid calibration and then frozen on two finer held-out meshes; the
	mixed-boundary test is reported only as an exploratory computation outside
	the smooth-domain theorem.
	
\end{enumerate}

The paper is organized as follows: preliminaries (Sec.~\ref{sec:prelim}),
the scheme (Sec.~\ref{sec:scheme}), properties (Sec.~\ref{sec:properties}),
convergence (Sec.~\ref{sec:convergence}), numerics (Sec.~\ref{sec:numerics}),
and conclusion (Sec.~\ref{sec7}).

	\section{Preliminaries}
	\label{sec:prelim}
	
	\subsection{Problem setting}
	
	The equation is written in backward form.  Let
	\[
	\cO_T:=[0,T)\times\cO,\qquad \overline{\cO}_T:=[0,T]\times\overline{\cO}.
	\]
	Consider the following Hamilton--Jacobi--Bellman equation with oblique Robin boundary conditions:
	\begin{equation}
		\begin{cases}
			-\partial_t u + H(t,x,Du,D^2u,u) = 0, & (t,x)\in\cO_T, \\[8pt]
			L(t,x,Du,u) = 0, & (t,x)\in[0,T)\times\partial\cO, \\[8pt]
			u(T,x) = \Psi(x), & x\in\overline{\cO},
		\end{cases}
		\label{eq:HJB-main}
	\end{equation}
	where
	\begin{equation}
		\label{eq:HJB-H}
		\begin{aligned}
			H(t,x,p,M,u) = \sup_{a\in A}\Big\{
			&-\dfrac{1}{2}\Tr\big(\sigma(t,x,a)\sigma(t,x,a)^{\top}M\big) \\
			&-\langle\mu(t,x,a),p\rangle
			+ r(t,x,a)\,u - f(t,x,a) \Big\},
		\end{aligned}
	\end{equation}
	\begin{equation}
		\label{eq:HJB-L}
		L(t,x,p,u) = \sup_{b\in B}\Big\{
		\langle\gamma(x,b),p\rangle + k(t,x,b)\,u - g(t,x,b) \Big\}.
 s	\end{equation}
	The data of the problem are as follows:
	\begin{itemize}
		\item $A\subset\R^{N_A}$, $B\subset\R^{N_B}$ are nonempty compact sets;
		\item $\sigma:[0,T]\times\overline{\cO}\times A\to\R^{N\times N_p}$ with $1\le N_p\le N$;
		\item $\mu:[0,T]\times\overline{\cO}\times A\to\R^N$;
		\item $f,r:[0,T]\times\overline{\cO}\times A\to\R$;
		\item $\gamma:\partial\cO\times\mathcal{V}\to\R^N$, where $\mathcal{V}\subseteq\R^{N_B}$ is an open set containing $B$;
		\item $g,k:[0,T]\times\partial\cO\times B\to\R$;
		\item $\Psi:\overline{\cO}\to\R$.
	\end{itemize}
	We denote $\gamma_b(x)=\gamma(x,b)$ for all $(x,b)\in\partial\cO\times B$.
	
	Equations~\eqref{eq:HJB-main}--\eqref{eq:HJB-L} extend the HJB equation with oblique derivative boundary conditions studied by Calzola--Carlini--Dupuis--Silva~\cite{CalzolaEtAl2023}.  The term $r(t,x,a)u$ is the ordinary-time discount, whereas $k(t,x,b)u$ is the oblique Robin term coupled to boundary local time and requires the separate boundary treatment developed below.  When $r\equiv0$ and $k\equiv0$, \eqref{eq:HJB-main}--\eqref{eq:HJB-L} reduce, up to the time reversal $t\mapsto T-t$, to the equation in~\cite{CalzolaEtAl2023}.

	\subsection{Basic assumptions}
	
	We make the following assumptions throughout the paper.

	\begin{enumerate}[
		label=\arabic*.,
		ref=\arabic*,
		labelsep=1em,
		leftmargin=*,
		itemindent=0pt
		]
		\item \label{ass:H1}
		$\mathcal{O}\subset\mathbb{R}^N$ is a nonempty bounded domain, and $\partial\mathcal{O}$ is of class $C^3$.
		
		\item \label{ass:H2}
		The functions $\sigma$, $\mu$, $f$, $g$, $r$, $k$ and $\Psi$ are continuous. Moreover, for each $a\in A$, the maps $\sigma(\cdot,\cdot,a)$ and $\mu(\cdot,\cdot,a)$ are Lipschitz continuous, with Lipschitz constants independent of $a\in A$.
		
		\item \label{ass:H3}
		The function $\gamma$ is of class $C^1$, and for every $(x,b)\in\partial\mathcal{O}\times B$,
		\begin{equation}
			|\gamma_b(x)|=1,\qquad \langle n(x),\gamma_b(x)\rangle > 0,
			\label{eq:gamma-cond}
		\end{equation}
		where $n(x)$ denotes the unit outward normal vector of $\overline{\mathcal{O}}$ at $x\in\partial\mathcal{O}$.
		
		\item \label{ass:H4}
		For all $(t,x,a)\in[0,T]\times\overline{\mathcal{O}}\times A$, $r(t,x,a)\ge 0$; and for all $(t,x,b)\in[0,T]\times\partial\mathcal{O}\times B$, $k(t,x,b)\ge 0$.
	\end{enumerate}
	The conditions in Assumptions~\ref{ass:H1}--\ref{ass:H3} on $\cO$, $\sigma$, $\mu$, $f$, $g$, $\Psi$, and $\gamma$ are those used in~\cite{CalzolaEtAl2023}; continuity of the added coefficients $r$ and $k$ is included in Assumption~\ref{ass:H2}.  Assumption~\ref{ass:H4} ensures that the attenuation factors used below do not exceed one; this property is used in the monotonicity and barrier arguments.

	\subsection{Viscosity solutions and comparison}
	
	The definition of viscosity solutions parallels that of~\cite{CalzolaEtAl2023}, with only the addition of the $ru$ term in the Hamiltonian and the $ku$ term in the boundary operator.
	
	For any bounded function $z:\overline{\cO}_T\to\R$, its upper (resp.\ lower) semicontinuous envelope $z^*$ (resp.\ $z_*$) is defined exactly as in~\cite{CalzolaEtAl2023} and is not repeated here.
	
	\begin{definition}[Viscosity subsolution / supersolution / solution] \upshape
		\label{def:viscosity}
		\hfill\strut
		\begin{enumerate}
			\item[\upshape(i)] An upper semicontinuous function $u_1:\overline{\cO}_T\to\R$ is called a \textbf{viscosity subsolution} of \eqref{eq:HJB-main} if, for every $(t,x)\in\overline{\cO}_T$ and every $\phi\in C^2(\overline{\cO}_T)$ such that $u_1-\phi$ attains a local maximum at $(t,x)$, the following holds:
			\begin{itemize}
				\item If $(t,x)\in\cO_T$,
				\begin{equation}
					-\partial_t\phi(t,x) + H\big(t,x,D\phi(t,x),D^2\phi(t,x),u_1(t,x)\big) \le 0;
					\label{eq:sub-interior}
				\end{equation}
				\item If $(t,x)\in[0,T)\times\partial\cO$,
				\begin{equation}
					\min\Big\{
					-\partial_t\phi(t,x) + H\big(t,x,D\phi(t,x),D^2\phi(t,x),u_1(t,x)\big),\;
					L\big(t,x,D\phi(t,x),u_1(t,x)\big)
					\Big\} \le 0;
					\label{eq:sub-boundary}
				\end{equation}
				\item If $(t,x)\in\{T\}\times\cO$,
				\begin{equation}
					\min\Big\{
					-\partial_t\phi(t,x) + H\big(t,x,D\phi(t,x),D^2\phi(t,x),u_1(t,x)\big),\;
					u_1(t,x)-\Psi(x)
					\Big\} \le 0;
					\label{eq:sub-terminal-interior}
				\end{equation}
				\item If $(t,x)\in\{T\}\times\partial\cO$,
				\begin{equation}
					\begin{split}
						\min\Big\{
						& -\partial_t\phi(t,x) + H\big(t,x,D\phi(t,x),D^2\phi(t,x),u_1(t,x)\big), \\
						& L\big(t,x,D\phi(t,x),u_1(t,x)\big),  u_1(t,x)-\Psi(x)
						\Big\} \le 0.
					\end{split}
					\label{eq:sub-terminal-boundary}
				\end{equation}
			\end{itemize}
			\item[\upshape(ii)] A lower semicontinuous function $u_2$ is called a \textbf{viscosity supersolution} if the inequalities in (i) are reversed: $\le$ is replaced by $\ge$, and $\min$ is replaced by $\max$, with $u_1$ replaced by $u_2$ everywhere.

			\item[\upshape(iii)] A bounded function $u$ is called a \textbf{viscosity solution} if its upper semicontinuous envelope $u^*$ is a viscosity subsolution and its lower semicontinuous envelope $u_*$ is a viscosity supersolution.
		\end{enumerate}
	\end{definition}
	\begin{remark} \upshape The $\min$ (resp.\ $\max$) form on the terminal layer $\{T\}\times\overline{\cO}$ can be replaced by the pointwise conditions $u_1(T,x)\le\Psi(x)$ (resp.\ $u_2(T,x)\ge\Psi(x)$); see~\cite[Proposition~6]{Bourgoing2008}.
	\end{remark}
	\begin{proposition}[Comparison principle] \upshape\label{thm:wellposed}
		Under Assumptions~\ref{ass:H1}--\ref{ass:H4}, let $v$ be a bounded
		upper semicontinuous viscosity subsolution and let $w$ be a bounded lower
		semicontinuous viscosity supersolution of~\eqref{eq:HJB-main}.  Then
		\[
		v\le w\qquad\text{on }\overline{\cO}_T.
		\]
	\end{proposition}
	\begin{proof}
		Compactness and Assumption~\ref{ass:H3} give
		\[
		\nu:=\min_{(x,b)\in\partial\cO\times B}
		n(x)\cdot\gamma(x,b)>0.
		\]
		Consequently, for every $\lambda\ge0$,
		\[
		L(t,x,p+\lambda n(x),z)
		\ge L(t,x,p,z)+\nu\lambda .
		\]
		Moreover, if $z_1\ge z_2$, then
		\[
		0\le H(t,x,p,M,z_1)-H(t,x,p,M,z_2)
		\le \|r\|_\infty(z_1-z_2),
		\]
		and
		\[
		0\le L(t,x,p,z_1)-L(t,x,p,z_2)
		\le \|k\|_\infty(z_1-z_2).
		\]
		Thus $H$ is continuous, proper, and degenerate elliptic, whereas $L$ is
		continuous, proper, uniformly oblique, and Lipschitz continuous in $p$.
		For example, if $|z|,|w|\le R$ and $d=|t-s|+|x-y|$, uniform moduli
		$\omega_k$ and $\omega_g$ give
		\begin{align*}
		&|L(t,x,p,z)-L(s,y,q,w)|\\
		&\qquad\le |p-q|+C_\gamma|q|\,|x-y|
		+\|k\|_\infty|z-w|+R\omega_k(d)+\omega_g(d).
		\end{align*}
		The uniform Lipschitz continuity of $\sigma$ and $\mu$, together with the
		uniform continuity of the remaining data on their compact domains, gives
		the analogous standard controlled-diffusion modulus for $H$.  Therefore,
		after reversing the
		time variable, the comparison theorem in
		\cite[Theorem~II.1]{Barles1993} applies.  The relaxed terminal formulation
		is equivalent to the pointwise terminal inequalities by
		\cite[Proposition~6]{Bourgoing2008}.
	\end{proof}
	\begin{remark} \upshape
		Existence is not assumed at this stage.  The half-relaxed-limit argument in
		Theorem~\ref{thm:convergence} constructs a viscosity subsolution and a
		viscosity supersolution.  The preceding comparison principle makes them
		coincide and thereby produces the unique continuous viscosity solution.
	\end{remark}
	
	\subsection{Oblique projection near the boundary}
	
	The results in this subsection are taken from~\cite{CalzolaEtAl2023} and are restated here for completeness; they are essential for the scheme construction in Section~3.
	
	A key geometric tool is the existence of a unique oblique projection onto $\partial\cO$ along the direction $\gamma_b$ for points sufficiently close to the boundary.  The following proposition is Proposition~1 of~\cite{CalzolaEtAl2023}, which extends the construction in~\cite[Section~1.2]{Gobet2001} to a control-dependent direction.
	
	\begin{proposition}[Oblique projection] \upshape\label{prop:oblique-proj}
		Under Assumptions~\ref{ass:H1} and~\ref{ass:H3}, there exists $R>0$ such that, for every $x\in\R^N$ satisfying $d(x,\partial\cO)<R$ and every $b\in B$, there exist unique $p^{\gamma_b}(x)\in\partial\cO$ and $d^{\gamma_b}(x)\in\R$ satisfying
		\begin{equation}
			x = p^{\gamma_b}(x) + d^{\gamma_b}(x)\,\gamma_b\!\big(p^{\gamma_b}(x)\big).
			\label{eq:oblique-decomp}
		\end{equation}
		The mappings
		\[
		(x,b)\longmapsto p^{\gamma_b}(x),\qquad (x,b)\longmapsto d^{\gamma_b}(x),
		\]
		called respectively the \textbf{oblique projection} onto $\partial\cO$ and the \textbf{signed algebraic distance} along the direction $\gamma_b$, are both of class $C^1$.
	\end{proposition}
	\begin{proof} \upshape See~\cite{CalzolaEtAl2023}, Proposition~1.\end{proof}
	
	\begin{lemma}[Uniform oblique-distance comparison]
		\upshape\label{lem:oblique-distance}
		After possibly decreasing the radius in
		Proposition~\ref{prop:oblique-proj}, there is a constant $C_\gamma>0$,
		independent of $b\in B$, such that
		\[
		|d^{\gamma_b}(x)|\le C_\gamma\dist(x,\partial\cO)
		\]
		whenever $x$ is in that tubular neighbourhood.  If in addition
		$x\notin\overline\cO$, then $d^{\gamma_b}(x)>0$ and
		\[
		p^{\gamma_b}(x)-d^{\gamma_b}(x)
		\gamma_b\bigl(p^{\gamma_b}(x)\bigr)\in\cO.
		\]
	\end{lemma}
	\begin{proof}
		Let $\rho$ be the signed Euclidean distance to $\partial\cO$, positive
		outside $\cO$ and negative inside.  Since $\partial\cO$ is of class $C^3$,
		$\rho$ is of class $C^2$ in a fixed two-sided tubular neighbourhood and
		$D\rho(p)=n(p)$ on $\partial\cO$.  Compactness and uniform obliqueness give
		\[
		\nu_0:=\min_{p\in\partial\cO,\,b\in B}
		n(p)\cdot\gamma_b(p)>0.
		\]
		On a smaller closed tubular neighbourhood, the joint $C^1$ regularity in
		Proposition~\ref{prop:oblique-proj} gives
		\(
		M:=\sup_{x,b}|D_xd^{\gamma_b}(x)|<\infty.
		\)
		If $\pi(x)$ is the normal projection of $x$ onto $\partial\cO$, then
		$d^{\gamma_b}(\pi(x))=0$.  The segment from $\pi(x)$ to $x$ remains in the
		tubular neighbourhood, and the mean-value theorem therefore yields
		\[
		|d^{\gamma_b}(x)|
		\le M|x-\pi(x)|=M\dist(x,\partial\cO),
		\]
		uniformly in $b$.
		
		It remains to verify the sign and the reflected location.  Taylor's formula,
		uniformly in $(p,b)\in\partial\cO\times B$, gives
		\[
		\rho\bigl(p+s\gamma_b(p)\bigr)
		=s\,n(p)\cdot\gamma_b(p)+O(s^2).
		\]
		Shrink the neighbourhood once more so that the quadratic remainder is at
		most $\nu_0|s|/2$.  For an exterior point write
		$x=p+d\gamma_b(p)$ using Proposition~\ref{prop:oblique-proj}.  Since
		$\rho(x)>0$, the preceding expansion implies $d>0$.  Applying it with
		$s=-d$ gives
		\[
		\rho\bigl(p-d\gamma_b(p)\bigr)
		\le-\nu_0d/2<0,
		\]
		so the reflected point belongs to $\cO$.
	\end{proof}
	
	For any $\varepsilon\ge0$, we introduce the following notation :
	\begin{equation}
		\begin{aligned}
			D_\varepsilon    &:= \{x\in\overline{\cO}\mid d(x,\partial\cO)>\varepsilon\},\\[2pt]
			\partial D_\varepsilon &:= \{x\in\overline{\cO}\mid d(x,\partial\cO)=\varepsilon\},\\[2pt]
			L_\varepsilon    &:= \{x\in\overline{\cO}\mid d(x,\partial\cO)\le\varepsilon\}.
		\end{aligned}
		\label{eq:domain-decomp}
	\end{equation}
	
	\begin{lemma} [Regularity of the distance function]. \upshape\label{lem:dist-reg}
		Under Assumption~\ref{ass:H1}, the following hold:
		\begin{enumerate}
			\item[\upshape(i)] There exists $\eta>0$ such that on $L_\eta$, the normal projection $p_{\partial\cO}$ is well defined and of class $C^1$.
			\item[\upshape(ii)] The distance function $L_\eta\ni x\mapsto d(x,\partial\cO)$ is of class $C^3$, and
			\[
			D d(\cdot,\partial\cO)(x) = -\,n\!\big(p_{\partial\cO}(x)\big).
			\]
			\item[\upshape(iii)] For every $\delta\in[0,\eta]$, $\partial D_\delta$ is of class $C^3$, and its unit outward normal is $n_\delta(x)=n(p_{\partial\cO}(x))$.
			\item[\upshape(iv)] For any $x\in L_\delta$ ($\delta\in[0,\eta]$), the point $p = p_{\partial\cO}(x) - \delta\,n(p_{\partial\cO}(x))$ is the projection of $x$ onto $\partial D_\delta$.
			\item[\upshape(v)] For every $\delta\in[0,\eta]$, the function $x\mapsto d(x,\partial D_\delta)$ is of class $C^3$ on $L_\delta$, and
			\begin{equation}
				d(x,\partial\cO) + d(x,\partial D_\delta) = \delta,\qquad \forall\,x\in L_\delta.
				\label{eq:dist-add}
			\end{equation}
		\end{enumerate}
	\end{lemma}
	\begin{proof}
		\upshape See~\cite[Lemma~1]{CalzolaEtAl2023}; parts (i) and (ii) are taken from~\cite[Lemma~14.16]{GilbargTrudinger2001}.
	\end{proof}

	\section{Probabilistic construction of the fully discrete scheme}
	\label{sec:scheme}
	
	\subsection{Spatial grid and \texorpdfstring{$\mathbb P_1$}{P1} interpolation}
	
	Let $\cO_h$ be the polyhedral approximation of $\cO$ and let $\mathbb T_h$
	be a fitted, shape-regular simplicial triangulation constructed as
	in~\cite[Section~3]{CalzolaEtAl2023}.  Its spatial mesh size is
	\[
	h:=\max_{\mathsf T\in\mathbb T_h}\operatorname{diam}(\mathsf T)
	\]
	and tends to zero along the mesh family.  Write
	\[
	\cG_h=\{x_i:i\in\mathcal I_h\}\subset\overline\cO,\qquad
	\mathcal I_h=\mathcal I_h^\circ\mathbin{\dot\cup}\mathcal I_h^\partial,
	\]
	where $x_i\in\cO$ for $i\in\mathcal I_h^\circ$ and
	$x_i\in\partial\cO$ for $i\in\mathcal I_h^\partial$.  As in the cited
	construction, $\widehat{\mathbb T}_h$ denotes the corresponding exact
	triangulation of $\overline\cO$ with the same vertices and with curved
	boundary elements.
	
	Let $\{\psi_j\}_{j\in\mathcal I_h}$ be the continuous piecewise-linear
	($\mathbb P_1$) nodal basis on $\mathbb T_h$, and let $p_h$ be the
	elementwise map from $\widehat{\mathbb T}_h$ to $\mathbb T_h$ used
	in~\cite[Section~3]{CalzolaEtAl2023}.  For a grid vector
	$V=(V_j)_{j\in\mathcal I_h}$ define
	\begin{equation}
		I_h[V](x):=\sum_{j\in\mathcal I_h}\psi_j(p_h(x))V_j,
		\qquad x\in\overline\cO.
		\label{eq:interpolation}
	\end{equation}
	This $\mathbb P_1$ interpolant is the only interpolation operator used below.
	Its barycentric weights are nonnegative and sum to one; hence $I_h$ is
	positivity preserving.  Moreover, the standard estimate proved for this
	mesh construction in~\cite[Lemma~2]{CalzolaEtAl2023} gives, for every
	$\phi\in C^2(\overline\cO)$,
	\begin{equation}
		\bigl\|I_h[\phi|_{\cG_h}]-\phi\bigr\|_{L^\infty(\overline\cO)}
		\le C_\phi h^2 .
		\label{eq:interp-error}
	\end{equation}
	
	Choose $N_T\ge1$ and set
	\[
	\dt=\frac{T}{N_T},\qquad t_n=n\dt,\qquad n=0,\ldots,N_T.
	\]
	\subsection{Reflected Feynman--Kac representation}
	
	We first suppress the controls in order to show where every term of the update comes from.  Consider
	\begin{equation}
		\begin{cases}
			-u_t-\dfrac12\Tr(\bar\sigma\bar\sigma^\top D^2u)-\bar\mu\cdot Du+\bar r u-\bar f=0,&(t,x)\in \cO_T,\\
			\bar\gamma\cdot Du+\bar k u=\bar g,&(t,x)\in[0,T)\times\partial\cO,\\
			u(T,x)=\Psi(x),&x\in\overline\cO.
		\end{cases}
		\label{eq:linear-problem}
	\end{equation}
	For this derivation, assume temporarily that a reflected semimartingale pair $(X,L)$ solving the equation below exists, that $u\in C^{1,2}([0,T]\times\overline\cO)$, and that the coefficients and $u$ have the boundedness and integrability needed for It\^{o}'s formula and conditional expectation.  These are auxiliary classical hypotheses only; the convergence proof below neither invokes the representation nor assumes this regularity.  Let
	\[
	dX_s=\bar\mu(s,X_s)\,ds+\bar\sigma(s,X_s)\,dW_s-\bar\gamma(X_s)\,dL_s,
	\]
	where $L$ is nondecreasing and increases only on $\partial\cO$.  The minus sign is consistent with $\bar\gamma$ pointing outward.  For $s\ge t_n$, define
	\[
	A_s=\exp\!\left(-\int_{t_n}^s\bar r(q,X_q)\,dq-\int_{t_n}^s\bar k(q,X_q)\,dL_q\right).
	\]
	It\^{o}'s formula gives the one-step conditional representation
	\begin{equation}
		\begin{aligned}
			u(t_n,x_i)=\E\Bigg[&A_{t_{n+1}}u(t_{n+1},X_{t_{n+1}})
			+\int_{t_n}^{t_{n+1}}A_s\bar f(s,X_s)\,ds\\
			&+\int_{t_n}^{t_{n+1}}A_s\bar g(s,X_s)\,dL_s
			\,\Big|\,X_{t_n}=x_i\Bigg].
		\end{aligned}
		\label{eq:FK}
	\end{equation}
	This identity, rather than a characteristic interpolation formula, is the starting point of the discrete method.
	
	\subsection{Equal-probability discretization of the linear representation}
	
	Set $m:=N_p$, the Brownian dimension.  Let $\xi\in\R^m$ have independent Rademacher components:
	\[
	\mathbb P(\xi_\ell=1)=\mathbb P(\xi_\ell=-1)=\frac12,
	\qquad \ell=1,\ldots,m.
	\]
	Its $P=2^m$ possible values are denoted by $\xi^1,\ldots,\xi^P$.  Every value occurs with probability $1/P$, and
	\begin{equation}
		\frac1P\sum_{p=1}^P\xi^p=0,
		\qquad
		\frac1P\sum_{p=1}^P\xi^p(\xi^p)^\top=I_m.
		\label{eq:cubature}
	\end{equation}
	These equal probabilities are part of the definition of the scheme and are not free numerical parameters.
	
	We first discretize the linear problem~\eqref{eq:linear-problem}.  Write
	$\overline\gamma=\gamma(\cdot,\overline b)$ for its fixed boundary
	parameter and abbreviate
	$p^{\overline\gamma}=p^{\gamma_{\overline b}}$ and
	$d^{\overline\gamma}=d^{\gamma_{\overline b}}$.  We retain the trial-point notation used later for the HJB
	scheme.  Its $p$th weak-Euler trial point is
	\begin{equation}
		X_{n,i}^{*,p}=x_i+\dt\,\overline\mu(t_n,x_i)
		+\sqrt{\dt}\,\overline\sigma(t_n,x_i)\xi^p.
		\label{eq:linear-trial}
	\end{equation}
	If $X_{n,i}^{*,p}\in\overline\cO$, put
	\[
	\widetilde X_{n,i}^p=X_{n,i}^{*,p},
	\qquad
	\widetilde d_{n,i}^p=\widetilde k_{n,i}^p=\widetilde g_{n,i}^p=0.
	\]
	If $X_{n,i}^{*,p}\notin\overline\cO$, let
	\begin{align}
		x_{n,i}^{\overline b,p}&=p^{\overline\gamma}(X_{n,i}^{*,p}),
		&\widetilde d_{n,i}^p&=2d^{\overline\gamma}(X_{n,i}^{*,p}),
		\label{eq:linear-projection}\\
		\widetilde X_{n,i}^p
		&=X_{n,i}^{*,p}-\widetilde d_{n,i}^p
		\overline\gamma(x_{n,i}^{\overline b,p}),
		\label{eq:linear-mirror}\\
		\widetilde k_{n,i}^p&=\overline k(t_n,x_{n,i}^{\overline b,p}),
		&\widetilde g_{n,i}^p&=\overline g(t_n,x_{n,i}^{\overline b,p}).
		\label{eq:linear-boundary-data}
	\end{align}
	The returned point $\widetilde X_{n,i}^p$ is the mirror image of the trial
	point about its oblique projection, and the round-trip distance
	$\widetilde d_{n,i}^p$ is used as the discrete increment of the regulator
	$L$ in~\eqref{eq:FK}.
	
	We now approximate the three terms in~\eqref{eq:FK} separately.  First freeze the ordinary-time coefficient at $(t_n,x_i)$.  The scalar discount is approximated positively by
	\begin{equation}
		e^{-\overline r(t_n,x_i)\dt}
		=\frac{1}{1+\overline r(t_n,x_i)\dt}+O(\dt^2).
		\label{eq:ordinary-discount}
	\end{equation}
	Freezing the boundary coefficient on branch $p$ then gives the continuation approximation
	\begin{equation}
		\frac{1}{1+\overline r(t_n,x_i)\dt}\,
		\frac1P\sum_{p=1}^P
		e^{-\widetilde k_{n,i}^p\widetilde d_{n,i}^p}
		u(t_{n+1},\widetilde X_{n,i}^p).
		\label{eq:linear-continuation}
	\end{equation}
	Second, the ordinary-time source is approximated by the left-endpoint rule,
	\begin{equation}
		\E\!\left[\int_{t_n}^{t_{n+1}}A_s\overline f(s,X_s)\,ds\right]
		\approx\dt\,\overline f(t_n,x_i).
		\label{eq:linear-volume-source}
	\end{equation}
	Third, on a branch with local-time increment $D$, freezing $K$ and $G$ gives
	\[
	\int_0^D e^{-K\ell}G\,d\ell.
	\]
	The midpoint rule yields
	\begin{equation}
		\int_0^D e^{-K\ell}G\,d\ell
		=D e^{-KD/2}G+O(D^3).
		\label{eq:midpoint-local-time}
	\end{equation}
	Since $\widetilde d_{n,i}^p=O(\sqrt\dt)$, the boundary-source contribution becomes
	\begin{equation}
		\frac{1}{1+\overline r(t_n,x_i)\dt}\,
		\frac1P\sum_{p=1}^P
		\widetilde d_{n,i}^p e^{-\widetilde k_{n,i}^p\widetilde d_{n,i}^p/2}
		\widetilde g_{n,i}^p,
		\label{eq:linear-boundary-source}
	\end{equation}
	up to the local $O(\dt^{3/2})$ quadrature error used later in the consistency proof.
	
	Replacing the next-time solution by a grid vector $V$ and using the $\mathbb P_1$ interpolant gives the linear interior operator
	\begin{equation}
		\begin{aligned}
			\overline{\cS}_{n,i}[V]
			=\frac{1}{1+\overline r(t_n,x_i)\dt}\frac1P\sum_{p=1}^P
			\Big[&e^{-\widetilde k_{n,i}^p\widetilde d_{n,i}^p}
			I_h[V](\widetilde X_{n,i}^p)\\
			&+\widetilde d_{n,i}^p
			e^{-\widetilde k_{n,i}^p\widetilde d_{n,i}^p/2}
			\widetilde g_{n,i}^p\Big]
			+\dt\,\overline f(t_n,x_i).
		\end{aligned}
		\label{eq:linear-scheme}
	\end{equation}
	For the boundary equation choose a purely spatial offset $\ell_h$, independent
	of $\dt$, such that $c_0h\le\ell_h\le c_1h$ for two fixed constants
	$0<c_0\le c_1$.  The complete linear method is
	\begin{equation}
		\begin{cases}
			U_i^n=\overline{\cS}_{n,i}[U^{n+1}],
			&n=N_T-1,\ldots,0,\quad i\in\mathcal I_h^\circ,\\[6pt]
			\displaystyle
			\frac{U_i^n-I_h[U^n](x_i-\ell_h\overline\gamma(x_i))}{\ell_h}
			+\overline k(t_n,x_i)U_i^n-\overline g(t_n,x_i)=0,
			&n=N_T-1,\ldots,0,\quad i\in\mathcal I_h^\partial,\\[9pt]
			U_i^{N_T}=\Psi(x_i),&i\in\mathcal I_h .
		\end{cases}
		\label{eq:complete-linear-scheme}
	\end{equation}
	Thus one first computes the interior values from $U^{n+1}$ and then solves
	the same-level boundary equations.  These equations can be implicit because
	the simplex containing $x_i-\ell_h\overline\gamma(x_i)$ may have boundary
	vertices.  Their coefficient matrix is a sparse nonsingular $M$-matrix by
	the $\mathbb P_1$ weight argument stated after the general scheme below.
	The coefficient $\overline k$ also acts on every exiting interior branch
	through local-time attenuation, while every branch probability remains
	exactly $1/P$.
	\subsection{Extension to the controlled HJB problem}
	
	We now return to~\eqref{eq:HJB-main}.  For one fixed evaluation of the interior
	update at $(t_n,x_i)$, choose $a\in A$ and $b\in B$ and freeze this pair over
	the corresponding one-step transition.  The same $b$ is used for all $P$
	equally probable branches emanating from $x_i$; a different boundary control
	is not selected on each branch.  After the $1/P$-weighted branch average has
	been formed, the infimum over $(a,b)$ is taken at that node.  Controls may
	therefore be selected anew at different nodes and time levels.  Using the
	notation of the original manuscript, the $p$th trial point is
	\begin{equation}
		X_{n,i}^{*,p}(a)=x_i+\dt\,\mu(t_n,x_i,a)
		+\sqrt{\dt}\,\sigma(t_n,x_i,a)\xi^p.
		\label{eq:trial}
	\end{equation}
	If $X_{n,i}^{*,p}(a)\in\overline\cO$, set
	\[
	\widetilde X_{n,i}^p(a,b)=X_{n,i}^{*,p}(a),\qquad
	\widetilde d_{n,i}^p(a,b)=\widetilde k_{n,i}^p(a,b)=
	\widetilde g_{n,i}^p(a,b)=0.
	\]
	If $X_{n,i}^{*,p}(a)\notin\overline\cO$, set
	\begin{align}
		x_{n,i}^{b,p}(a)&=p^{\gamma_b}(X_{n,i}^{*,p}(a)),\\
		\widetilde d_{n,i}^p(a,b)&=2d^{\gamma_b}(X_{n,i}^{*,p}(a)),\\
		\widetilde X_{n,i}^p(a,b)&=X_{n,i}^{*,p}(a)
		-\widetilde d_{n,i}^p(a,b)\gamma_b(x_{n,i}^{b,p}(a)),\\
		\widetilde k_{n,i}^p(a,b)&=k(t_n,x_{n,i}^{b,p}(a),b),\\
		\widetilde g_{n,i}^p(a,b)&=g(t_n,x_{n,i}^{b,p}(a),b).
		\label{eq:branch-data}
	\end{align}
	Proposition~\ref{prop:oblique-proj} makes these quantities well defined for sufficiently small $\dt$.  Uniformly in all nodes and controls,
	\begin{equation}
		|X_{n,i}^{*,p}(a)-x_i|+\widetilde d_{n,i}^p(a,b)\le C\sqrt{\dt}.
		\label{eq:step-bound}
	\end{equation}
	Indeed, continuity on the compact coefficient set and finiteness of the
	Rademacher family give, for $\dt\le1$,
	\[
	|X_{n,i}^{*,p}(a)-x_i|
	\le\dt\|\mu\|_\infty
	+\sqrt\dt\|\sigma\|_\infty\max_{1\le q\le P}|\xi^q|
	\le C_0\sqrt\dt .
	\]
	If the trial point is exterior, the segment joining $x_i\in\cO$ to
	$X_{n,i}^{*,p}(a)$ meets $\partial\cO$, and hence
	\[
	\dist(X_{n,i}^{*,p}(a),\partial\cO)
	\le|X_{n,i}^{*,p}(a)-x_i|.
	\]
	For sufficiently small $\dt$, Lemma~\ref{lem:oblique-distance} applies and
	gives
	\[
	0<\widetilde d_{n,i}^p(a,b)
	=2d^{\gamma_b}(X_{n,i}^{*,p}(a))
	\le2C_\gamma|X_{n,i}^{*,p}(a)-x_i|
	\le C\sqrt\dt.
	\]
	The same lemma shows that $\widetilde X_{n,i}^p(a,b)\in\cO$.  For a
	non-exiting branch $\widetilde d_{n,i}^p=0$ and
	$\widetilde X_{n,i}^p=X_{n,i}^{*,p}\in\overline\cO$.  Thus
	\eqref{eq:step-bound} holds and every interpolation value below is well
	defined.
	
	Applying the already derived linear recurrence with the frozen coefficients gives, for a grid vector $V$, the following operator.  In its defining sum only, we suppress the common arguments $(n,i,a,b)$:
	\begin{equation}
		\begin{aligned}
			\cS_{n,i}^{a,b}[V]
			=\frac{1}{1+r(t_n,x_i,a)\dt}\frac1P\sum_{p=1}^P
			\Big[e^{-\widetilde k_p\widetilde d_p}I_h[V](\widetilde X_p)
			+\widetilde d_p e^{-\widetilde k_p\widetilde d_p/2}\widetilde g_p\Big]
			+\dt f(t_n,x_i,a).
		\end{aligned}
		\label{eq:fixed-operator}
	\end{equation}
	For $i\in\mathcal I_h^\circ$, the HJB operator minimizes the fixed-control conditional expectation:
	\begin{equation}
		S_{n,i}[V]=\inf_{(a,b)\in A\times B}\cS_{n,i}^{a,b}[V].
		\label{eq:HJB-operator}
	\end{equation}
	
	For $i\in\mathcal I_h^\partial$ and $b\in B$, set
	\[
	y_{i,b}:=x_i-\ell_h\gamma(x_i,b).
	\]
	Uniform obliqueness implies $y_{i,b}\in\cO$ for all sufficiently small $h$,
	uniformly in $i$ and $b$.  Set
	\[
	\omega_{ij}^b:=\psi_j(p_h(y_{i,b})).
	\]
	The fixed-control and controlled boundary updates are
	\begin{align}
		\cS_{n,i}^{\partial,b}[V]
		&:=\frac{I_h[V](y_{i,b})+\ell_h\,g(t_n,x_i,b)}
		{1+\ell_h\,k(t_n,x_i,b)},
		\label{eq:fixed-boundary-map}\\
		S_{n,i}^{\partial}[V]&:=\inf_{b\in B}\cS_{n,i}^{\partial,b}[V].
		\label{eq:controlled-boundary-map}
	\end{align}
	Since $1+\ell_h\,k(t_n,x_i,b)>0$ and $B$ is compact,
	$V_i=S_{n,i}^{\partial}[V]$ is equivalent to
	\begin{equation}
		\sup_{b\in B}
		\left\{
		\frac{V_i-I_h[V](x_i-\ell_h\gamma(x_i,b))}{\ell_h}
		+k(t_n,x_i,b)V_i-g(t_n,x_i,b)
		\right\}=0.
		\label{eq:discrete-boundary-residual}
	\end{equation}
	
	The complete nonlinear method is therefore
	\begin{equation}
		\begin{cases}
			U_i^n=S_{n,i}[U^{n+1}],
			&n=N_T-1,\ldots,0,\quad i\in\mathcal I_h^\circ,\\[7pt]
			\begin{aligned}
				\sup_{b\in B}\biggl\{&
				\frac{U_i^n-I_h[U^n](x_i-\ell_h\gamma(x_i,b))}{\ell_h}\\[-2pt]
				&+k(t_n,x_i,b)U_i^n-g(t_n,x_i,b)
				\biggr\}=0,
			\end{aligned}
			&n=N_T-1,\ldots,0,\quad i\in\mathcal I_h^\partial,\\[10pt]
			U_i^{N_T}=\Psi(x_i),&i\in\mathcal I_h .
		\end{cases}
		\label{eq:scheme}
	\end{equation}
	
	\begin{lemma}[Uniform interior interpolation weight]
		\label{lem:uniform-interior-weight}
		There exist $h_0>0$ and $\theta\in(0,1)$, independent of $h$, $i$, and $b$,
		such that, for $0<h\le h_0$,
		\begin{equation}
			\sum_{j\in\mathcal I_h^\circ}\omega_{ij}^b\ge\theta .
			\label{eq:boundary-row-margin}
		\end{equation}
	\end{lemma}
	\begin{proof}
		Let
		\[
		\nu:=\min_{(x,b)\in\partial\cO\times B}
		n(x)\cdot\gamma(x,b)>0,
		\]
		and choose $\varrho\in C^2(\overline\cO)$ that agrees with
		$\dist(\cdot,\partial\cO)$ in a fixed inner tubular neighbourhood.  Since
		$D\varrho=-n$ on $\partial\cO$, Taylor's formula and
		$c_0h\le\ell_h\le c_1h$ give, uniformly in $i$ and $b$,
		\[
		\varrho(y_{i,b})
		=\ell_h n(x_i)\cdot\gamma(x_i,b)+O(\ell_h^2)
		\ge c_0\nu h-Ch^2.
		\]
		The interpolation estimate~\eqref{eq:interp-error} therefore implies, for
		all sufficiently small $h$,
		\[
		I_h[\varrho|_{\cG_h}](y_{i,b})\ge\frac{c_0\nu}{2}h.
		\]
		
		Let $\mathsf T\in\mathbb T_h$ be an active simplex for $p_h(y_{i,b})$.
		If $\mathsf T$ has no boundary vertex, every nonzero barycentric weight is
		attached to an interior node and the sum in~\eqref{eq:boundary-row-margin}
		equals one.  Otherwise choose a boundary vertex
		$x_q\in\mathsf T\cap\partial\cO$.  Every vertex $x_j$ of $\mathsf T$
		satisfies
		\[
		\dist(x_j,\partial\cO)\le|x_j-x_q|
		\le\operatorname{diam}(\mathsf T)\le h.
		\]
		For small $h$, $\varrho(x_j)=\dist(x_j,\partial\cO)$ at all these vertices,
		and $\varrho$ vanishes at the boundary vertices.  Locality and nonnegativity
		of the $\mathbb P_1$ weights now give
		\[
		I_h[\varrho|_{\cG_h}](y_{i,b})
		=\sum_j\omega_{ij}^b\varrho(x_j)
		\le h\sum_{j\in\mathcal I_h^\circ}\omega_{ij}^b.
		\]
		The last two estimates prove~\eqref{eq:boundary-row-margin}, for example
		with $\theta=\min\{1/2,c_0\nu/2\}$.  This is precisely where the lower
		offset bound $\ell_h\ge c_0h$ is used.
	\end{proof}
	
	After the interior values at time level $t_n$ have been computed, regard the
	right-hand side of~\eqref{eq:controlled-boundary-map} as a map $T_n$ of the
	boundary vector.  Lemma~\ref{lem:uniform-interior-weight}, $k\ge0$, and the
	partition-of-unity property give directly
	\[
	\|T_nz-T_n\widetilde z\|_\infty
	\le\sup_{i,b}
	\frac{\sum_{j\in\mathcal I_h^\partial}\omega_{ij}^b}
	{1+\ell_hk(t_n,x_i,b)}
	\|z-\widetilde z\|_\infty
	\le(1-\theta)\|z-\widetilde z\|_\infty.
	\]
	Hence the boundary values are obtained by the contraction iteration
	$z^{(q+1)}=T_nz^{(q)}$.  In the linear case the same row-margin calculation
	shows directly that the boundary coefficient matrix is a sparse nonsingular
	$M$-matrix.  Indeed, writing $k_i=k(t_n,x_i)$ for fixed linear data, its
	boundary block has entries
	\[
	M_{ii}=1+\ell_h k_i-\omega_{ii},\qquad
	M_{ij}=-\omega_{ij}\quad(i\ne j),
	\]
	and its strict row margin is
	\[
	M_{ii}-\sum_{j\ne i}|M_{ij}|
	=\ell_h k_i+\sum_{j\in\mathcal I_h^\circ}\omega_{ij}
	\ge\theta.
	\]
	\begin{remark}[Coupling of the two Robin weights]\label{rem:coupled-weights}
		Let $x_b\in\partial\cO$,
		$X^*=x_b+(\widetilde d/2)\gamma$, and
		$\widetilde X=x_b-(\widetilde d/2)\gamma$.  For a smooth test function,
		Taylor expansion gives
		\[
		e^{-k\widetilde d}\phi(\widetilde X)
		+\widetilde d e^{-k\widetilde d/2}g-\phi(X^*)
		=-\widetilde d\bigl(\gamma\cdot D\phi+k\phi-g\bigr)(x_b)
		+\frac{k\widetilde d^2}{2}
		\bigl(\gamma\cdot D\phi+k\phi-g\bigr)(x_b)
		+O(\widetilde d^3).
		\]
		Thus $e^{-k\widetilde d}$ is not a modification of the branch probability $1/P$; it is
		the local-time attenuation of the continuation value.  The midpoint factor in
		the boundary source makes the entire quadratic term a multiple of the same
		Robin residual.  This elementary identity is used explicitly in the
		near-boundary branch expansion below.
	\end{remark}
	
	\begin{remark}[Polygonal implementations]
		The convergence theorem below is proved for a $C^3$ boundary, where the control-uniform oblique projection is available.  In the disk experiments the branch projection and reflection use the exact circle, while interpolation is performed on a fitted regular polygon after an $O(h^2)$ radial cap map.  On a rectangle, a numerical branch can be reflected face by face using the constant outward normal of each face.  Mixed Robin--Dirichlet stopping requires a separate first-hit rule and is reported only as a numerical extension; it is not used as evidence for the smooth-domain convergence theorem.
	\end{remark}
	
	The role of the Robin term is now unambiguous.  When $\widetilde k\widetilde d>0$, the coefficient of the unknown in an exiting branch is $e^{-\widetilde k\widetilde d}<1$.  No change of the prescribed source alone can reproduce this coefficient for every continuation vector.  This observation distinguishes the local-time attenuation from an additive modification of an oblique-derivative boundary update, without changing the equal probability assigned to the branch.
	
	\section{Properties of the fully discrete scheme}
	\label{sec:properties}
	
	For $(t,x,p,M,z,a)\in\overline{\cO}_T\times\R^N\times\R^{N\times N}\times\R\times A$
	and $(t,x,p,z,b)\in[0,T]\times\partial\cO\times\R^N\times\R\times B$, define
	the control-resolved operators
	\begin{equation}
		\begin{aligned}
			\mathcal H(t,x,p,M,z,a)
			&:=-\frac12\Tr\!\left(\sigma\sigma^\top(t,x,a)M\right)
			-\mu(t,x,a)\cdot p+r(t,x,a)z-f(t,x,a),\\
			\mathcal L(t,x,p,z,b)
			&:=\gamma(x,b)\cdot p+k(t,x,b)z-g(t,x,b).
		\end{aligned}
		\label{eq:calH-calL}
	\end{equation}
	Thus $H=\sup_{a\in A}\mathcal H$ and $L=\sup_{b\in B}\mathcal L$.  As in the
	original consistency calculation, set
	\begin{equation}
		\mathcal R(t,x,p,z,b)
		:=\frac12 k(t,x,b)^2z-\frac12 k(t,x,b)g(t,x,b)
		+\frac12 k(t,x,b)\gamma(x,b)\cdot p.
		\label{eq:calR}
	\end{equation}
	In particular,
	\begin{equation}
		\mathcal R(t,x,p,z,b)=\frac12 k(t,x,b)\mathcal L(t,x,p,z,b).
		\label{eq:R-equals-kL}
	\end{equation}
	For an exiting branch let
	$x_b^p:=x_{n,i}^{b,p}(a)=p^{\gamma_b}(X_{n,i}^{*,p}(a))\in\partial\cO$ and define
	\begin{equation}
		\widetilde{\mathcal L}_{n,i}^p(q,v,a,b):=
		\begin{cases}
			0,&X_{n,i}^{*,p}(a)\in\overline\cO,\\[3pt]
			\mathcal L(t_n,x_b^p,q,v,b),&X_{n,i}^{*,p}(a)\notin\overline\cO,
		\end{cases}
		\label{eq:tildeL}
	\end{equation}
	and define $\widetilde{\mathcal R}_{n,i}^p$ in the same way with
	$\mathcal L$ replaced by $\mathcal R$.  On a non-exiting branch the displayed
	arguments involving $x_b^p$ are ignored because both truncated operators are
	defined to be zero.
	
	For a smooth function $\phi$, we use the original shorthand
	\[
	\bigl[-\partial_t\phi+\mathcal H\bigr]
	(t,x,D\phi,D^2\phi,\phi,a)
	:=-\partial_t\phi(t,x)
	+\mathcal H(t,x,D\phi(t,x),D^2\phi(t,x),\phi(t,x),a).
	\]
	
	\subsection{Monotonicity of the parabolic scheme}
	
	\begin{proposition}[Parabolic monotonicity]\label{prop:monotone}
		Under Assumptions~\ref{ass:H1}--\ref{ass:H4}, let $V\le W$ be two grid
		functions.  For every $n=0,\ldots,N_T-1$ the following assertions hold:
		\begin{align}
			\cS_{n,i}^{a,b}[V]&\le\cS_{n,i}^{a,b}[W],
			&S_{n,i}[V]&\le S_{n,i}[W],
			&&i\in\mathcal I_h^\circ, (a,b)\in A\times B,
			\label{eq:para-monotone-interior}\\
			\cS_{n,i}^{\partial,b}[V]&\le\cS_{n,i}^{\partial,b}[W],
			&S_{n,i}^{\partial}[V]&\le S_{n,i}^{\partial}[W],
			&&i\in\mathcal I_h^\partial, b\in B.
			\label{eq:para-monotone-boundary}
		\end{align}
		Consequently, one complete backward time-layer update is order preserving.
		Moreover, for every constant $c\ge0$,
		\begin{equation}
			S_{n,i}[V+c]\le S_{n,i}[V]+c,\qquad
			S_{n,i}^{\partial}[V+c]\le S_{n,i}^{\partial}[V]+c.
			\label{eq:para-shift}
		\end{equation}
		If $f,g,$ and $\Psi$ are nonnegative, then the discrete solution is
		nonnegative.
	\end{proposition}
	
	\begin{proof}
		All assertions follow directly from~\eqref{eq:scheme}: the interpolation and branch coefficients are nonnegative, their continuation coefficients have total mass at most one, infima preserve order, and the same-level boundary equation is solved by an order-preserving contraction.
	\end{proof}
	
	For completeness, the constant in the shift estimate is explicit.  Put
	\[
	C_{rk}:=\max\{\|r\|_\infty,\|k\|_\infty\}.
	\]
	For every
	$i\in\mathcal I_h^\circ$, fixed $(a,b)$, and constant $c$, adding $c$ to
	$V$ multiplies it by
	\[
	\alpha_{n,i}^{a,b}:=
	\frac{P^{-1}\sum_{p=1}^P
		e^{-\widetilde k_{n,i}^p(a,b)\widetilde d_{n,i}^p(a,b)}}
	{1+r(t_n,x_i,a)\dt}\in[0,1].
	\]
	Since $1-e^{-z}\le z$,
	\begin{equation}
		\left|\cS_{n,i}^{a,b}[V+c]-\cS_{n,i}^{a,b}[V]-c\right|
		=|c|(1-\alpha_{n,i}^{a,b})
		\le C_{rk}|c|\left(\dt+\frac1P\sum_{p=1}^P
		\widetilde d_{n,i}^p(a,b)\right).
		\label{eq:shift-defect}
	\end{equation}
	
	\subsection{Consistency of the parabolic scheme}
	
	\begin{proposition}[Parabolic consistency]\label{prop:nodewise-consistency}
		Let $\phi$ be the restriction to $[0,T]\times\overline\cO$ of a $C^3$
		function defined on a fixed neighbourhood, with all derivatives used below
		bounded there, and put $\phi_{n,i}:=\phi(t_n,x_i)$.  For sufficiently small
		$\Delta t$ and $h$, the following statements hold uniformly in the indicated nodes
		and controls.
		
		\smallskip
		\noindent\textup{(I) Interior nodes.}
		For $i\in\mathcal I_h^\circ$, $a\in A$, and $b\in B$,
		\begin{align}
			&\cS_{n,i}^{a,b}\bigl[\phi(t_{n+1},\cdot)|_{\cG_h}\bigr]-\phi_{n,i}
			\notag\\
			&\quad=-\Delta t\,
			\bigl[-\partial_t\phi+\mathcal H\bigr]
			\bigl(t_n,x_i,D\phi_{n,i},D^2\phi_{n,i},\phi_{n,i},a\bigr)
			\notag\\
			&\qquad-\frac1P\sum_{p=1}^P\widetilde d_{n,i}^p(a,b)
			\widetilde{\mathcal L}_{n,i}^p
			\bigl(D\phi(t_n,x_b^p),\phi(t_n,x_b^p),a,b\bigr)
			\notag\\
			&\qquad+\frac1P\sum_{p=1}^P\bigl(\widetilde d_{n,i}^p(a,b)\bigr)^2
			\widetilde{\mathcal R}_{n,i}^p
			\bigl(D\phi(t_n,x_b^p),\phi(t_n,x_b^p),a,b\bigr)
			+O(\Delta t^{3/2}+h^2).
			\label{eq:reflected-expansion-fixed}
		\end{align}
		Consequently,
		\begin{align}
			&S_{n,i}\bigl[\phi(t_{n+1},\cdot)|_{\cG_h}\bigr]-\phi_{n,i}
			\notag\\
			&\quad=-\sup_{a\in A,\,b\in B}\Biggl\{
			\Delta t\,
			\bigl[-\partial_t\phi+\mathcal H\bigr]
			\bigl(t_n,x_i,D\phi_{n,i},D^2\phi_{n,i},\phi_{n,i},a\bigr)
			\notag\\
			&\hspace{25mm}+\frac1P\sum_{p=1}^P\widetilde d_{n,i}^p(a,b)
			\widetilde{\mathcal L}_{n,i}^p
			\bigl(D\phi(t_n,x_b^p),\phi(t_n,x_b^p),a,b\bigr)
			\notag\\
			&\hspace{25mm}-\frac1P\sum_{p=1}^P\bigl(\widetilde d_{n,i}^p(a,b)\bigr)^2
			\widetilde{\mathcal R}_{n,i}^p
			\bigl(D\phi(t_n,x_b^p),\phi(t_n,x_b^p),a,b\bigr)
			\Biggr\}
			+O(\Delta t^{3/2}+h^2).
			\label{eq:reflected-expansion-controlled}
		\end{align}
		
		\smallskip
		\noindent\textup{(II) Boundary nodes.}
		For $i\in\mathcal I_h^\partial$,
		\begin{equation}
			S_{n,i}^{\partial}\bigl[\phi(t_n,\cdot)|_{\cG_h}\bigr]-\phi_{n,i}
			=-\ell_h\sup_{b\in B}
			\mathcal L(t_n,x_i,D\phi_{n,i},\phi_{n,i},b)+O(h^2).
			\label{eq:boundary-controlled-consistency}
		\end{equation}
		The boundary formula is entirely at time $t_n$; it contains neither a time
		derivative nor the volume Hamiltonian.
	\end{proposition}
	
	\begin{proof}
		Throughout the proof, $C>0$ denotes a generic constant independent of the
		discretization parameters.  Fix $n$, $i\in\mathcal I_h^\circ$, $a\in A$,
		$b\in B$, and put
		\[
		\mu=\mu(t_n,x_i,a),\quad \sigma=\sigma(t_n,x_i,a),\quad
		r=r(t_n,x_i,a),\quad f=f(t_n,x_i,a).
		\]
		
		\medskip
		\noindent\textbf{Step 1: A single random-walk branch.}
		Fix $p\in\{1,\ldots,P\}$ and write
		\[
		X^*:=x_i+\Delta t\,\mu+\sqrt{\Delta t}\,\sigma\xi^p.
		\]
		
		\smallskip
		\noindent\emph{Case A: $X^*\in\overline\cO$.}
		Here $\widetilde d=0$ and $\widetilde X=X^*$.  Taylor expansion about
		$(t_n,x_i)$ yields
		\begin{align}
			\phi(t_{n+1},\widetilde X)
			={}&\phi_{n,i}+\Delta t\,\partial_t\phi_{n,i}
			+\Delta t\,D\phi_{n,i}\cdot\mu
			+\sqrt{\Delta t}\,D\phi_{n,i}\cdot\sigma\xi^p
			\notag\\
			&+\frac{\Delta t}{2}(\sigma\xi^p)^\top
			D^2\phi_{n,i}(\sigma\xi^p)+O(\Delta t^{3/2}).
			\label{eq:cons-caseA}
		\end{align}
		
		\smallskip
		\noindent\emph{Case B: $X^*\notin\overline\cO$.}
		Let $x_b^p=p^{\gamma_b}(X^*)$, set
		$\widetilde\gamma=\gamma(x_b^p,b)$ and
		$\widetilde d=2d^{\gamma_b}(X^*)$.  Then
		\begin{equation}
			\widetilde X=x_b^p-\frac{\widetilde d}{2}\widetilde\gamma,
			\qquad
			X^*=x_b^p+\frac{\widetilde d}{2}\widetilde\gamma,
			\qquad \widetilde d=O(\sqrt{\Delta t}).
			\label{eq:cons-X-tilde}
		\end{equation}
		Write $\widetilde k=k(t_n,x_b^p,b)$ and
		$\widetilde g=g(t_n,x_b^p,b)$.  Taylor expansion about $(t_n,x_b^p)$ gives
		\begin{align}
			\phi(t_{n+1},\widetilde X)
			={}&\phi(t_n,x_b^p)+\Delta t\,\partial_t\phi(t_n,x_b^p)
			-\frac{\widetilde d}{2}D\phi(t_n,x_b^p)\cdot\widetilde\gamma
			\notag\\
			&+\frac{\widetilde d^2}{8}\widetilde\gamma^\top
			D^2\phi(t_n,x_b^p)\widetilde\gamma+O(\Delta t^{3/2}),
			\label{eq:cons-xb-expand}
		\end{align}
		and
		\begin{equation}
			e^{-\widetilde k\widetilde d}
			=1-\widetilde k\widetilde d+\frac12\widetilde k^2\widetilde d^2
			+O(\widetilde d^3),\qquad
			\widetilde d e^{-\widetilde k\widetilde d/2}
			=\widetilde d-\frac12\widetilde k\widetilde d^2+O(\widetilde d^3).
			\label{eq:cons-exp-expand}
		\end{equation}
		Substitution and collection of powers of $\widetilde d$ yield
		\begin{align}
			&e^{-\widetilde k\widetilde d}\phi(t_{n+1},\widetilde X)
			+\widetilde d e^{-\widetilde k\widetilde d/2}\widetilde g
			\notag\\
			&\quad=\phi(t_n,x_b^p)+\Delta t\,\partial_t\phi(t_n,x_b^p)
			+\widetilde d\left(\widetilde g-\widetilde k\phi(t_n,x_b^p)
			-\frac12D\phi(t_n,x_b^p)\cdot\widetilde\gamma\right)
			\notag\\
			&\qquad+\widetilde d^2\left(
			\frac12\widetilde k^2\phi(t_n,x_b^p)-\frac12\widetilde k\widetilde g
			+\frac12\widetilde kD\phi(t_n,x_b^p)\cdot\widetilde\gamma
			+\frac18\widetilde\gamma^\top D^2\phi(t_n,x_b^p)\widetilde\gamma
			\right)
			+O(\Delta t^{3/2}).
			\label{eq:cons-combined}
		\end{align}
		Adding and subtracting the Taylor terms that move from
		$x_b^p-\widetilde d\widetilde\gamma/2$ to
		$x_b^p+\widetilde d\widetilde\gamma/2$ gives
		\begin{align}
			&e^{-\widetilde k\widetilde d}\phi(t_{n+1},\widetilde X)
			+\widetilde d e^{-\widetilde k\widetilde d/2}\widetilde g
			\notag\\
			&\quad=\Bigl[\phi(t_n,x_b^p)+\Delta t\,\partial_t\phi(t_n,x_b^p)
			+\frac{\widetilde d}{2}D\phi(t_n,x_b^p)\cdot\widetilde\gamma
			+\frac{\widetilde d^2}{8}\widetilde\gamma^\top
			D^2\phi(t_n,x_b^p)\widetilde\gamma\Bigr]
			\notag\\
			&\qquad-\widetilde d\,
			\mathcal L(t_n,x_b^p,D\phi(t_n,x_b^p),\phi(t_n,x_b^p),b)
			\notag\\
			&\qquad+\widetilde d^2\,
			\mathcal R(t_n,x_b^p,D\phi(t_n,x_b^p),\phi(t_n,x_b^p),b)
			+O(\Delta t^{3/2}).
			\label{eq:cons-extract}
		\end{align}
		
		It remains to express the bracketed Taylor polynomial at the starting point.
		From~\eqref{eq:cons-X-tilde},
		\begin{equation}
			x_b^p-x_i=\Delta t\,\mu+\sqrt{\Delta t}\,\sigma\xi^p
			-\frac{\widetilde d}{2}\widetilde\gamma.
			\label{eq:xb-xi-relation}
		\end{equation}
		Set $\delta=x_b^p-x_i$.  Taylor expansion around $x_i$ gives
		\begin{align*}
			\phi(t_n,x_b^p)
			&=\phi_{n,i}+D\phi_{n,i}\cdot\delta
			+\frac12\delta^\top D^2\phi_{n,i}\delta+O(|\delta|^3),\\
			\Delta t\,\partial_t\phi(t_n,x_b^p)
			&=\Delta t\,\partial_t\phi_{n,i}+O(\Delta t|\delta|),\\
			\frac{\widetilde d}{2}D\phi(t_n,x_b^p)\cdot\widetilde\gamma
			&=\frac{\widetilde d}{2}D\phi_{n,i}\cdot\widetilde\gamma
			+\frac{\widetilde d}{2}(D^2\phi_{n,i}\delta)\cdot\widetilde\gamma
			+O(\widetilde d|\delta|^2),\\
			\frac{\widetilde d^2}{8}\widetilde\gamma^\top
			D^2\phi(t_n,x_b^p)\widetilde\gamma
			&=\frac{\widetilde d^2}{8}\widetilde\gamma^\top
			D^2\phi_{n,i}\widetilde\gamma+O(\widetilde d^2|\delta|).
		\end{align*}
		To see the cancellation explicitly, put
		$q=\Delta t\,\mu+\sqrt{\Delta t}\,\sigma\xi^p$, so that
		$\delta=q-\widetilde d\widetilde\gamma/2$.  Then
		\[
		D\phi_{n,i}\cdot\delta
		+\frac{\widetilde d}{2}D\phi_{n,i}\cdot\widetilde\gamma
		=D\phi_{n,i}\cdot q,
		\]
		and, with $H=D^2\phi_{n,i}$,
		\[
		\frac12\delta^\top H\delta
		+\frac{\widetilde d}{2}(H\delta)\cdot\widetilde\gamma
		+\frac{\widetilde d^2}{8}\widetilde\gamma^\top H\widetilde\gamma
		=\frac12q^\top Hq.
		\]
		The drift--drift part of $q^\top Hq$ is $O(\Delta t^2)$ and its mixed
		drift--diffusion part is $O(\Delta t^{3/2})$.  We therefore obtain
		\begin{align}
			&e^{-\widetilde k\widetilde d}\phi(t_{n+1},\widetilde X)
			+\widetilde d e^{-\widetilde k\widetilde d/2}\widetilde g
			\notag\\
			&\quad=\phi_{n,i}+\Delta t\,\partial_t\phi_{n,i}
			+\Delta t\,D\phi_{n,i}\cdot\mu
			+\sqrt{\Delta t}\,D\phi_{n,i}\cdot\sigma\xi^p
			\notag\\
			&\qquad+\frac{\Delta t}{2}(\sigma\xi^p)^\top
			D^2\phi_{n,i}(\sigma\xi^p)
			\notag\\
			&\qquad-\widetilde d_{n,i}^p(a,b)
			\widetilde{\mathcal L}_{n,i}^p
			\bigl(D\phi(t_n,x_b^p),\phi(t_n,x_b^p),a,b\bigr)
			\notag\\
			&\qquad+\bigl(\widetilde d_{n,i}^p(a,b)\bigr)^2
			\widetilde{\mathcal R}_{n,i}^p
			\bigl(D\phi(t_n,x_b^p),\phi(t_n,x_b^p),a,b\bigr)
			+O(\Delta t^{3/2}).
			\label{eq:cons-caseB-final}
		\end{align}
		For a non-exiting branch, the last two lines vanish by definition and
		\eqref{eq:cons-caseB-final} reduces to~\eqref{eq:cons-caseA}.  Thus
		\eqref{eq:cons-caseB-final} is valid for every branch.
		
		\medskip
		\noindent\textbf{Step 2: Average the $P$ equally probable branches.}
		Using
		\[
		\frac1P\sum_{p=1}^P\xi^p=0,
		\qquad
		\frac1P\sum_{p=1}^P\xi^p(\xi^p)^\top=I_m,
		\]
		the linear stochastic term vanishes and the quadratic stochastic term
		becomes a trace.  Therefore
		\begin{align}
			&\frac1P\sum_{p=1}^P\Bigl[
			e^{-\widetilde k_{n,i}^p\widetilde d_{n,i}^p}
			\phi(t_{n+1},\widetilde X_{n,i}^p)
			+\widetilde d_{n,i}^p
			e^{-\widetilde k_{n,i}^p\widetilde d_{n,i}^p/2}
			\widetilde g_{n,i}^p\Bigr]
			\notag\\
			&\quad=\phi_{n,i}+\Delta t\,\partial_t\phi_{n,i}
			+\Delta t\,\mu\cdot D\phi_{n,i}
			+\frac{\Delta t}{2}\Tr(\sigma\sigma^\top D^2\phi_{n,i})
			\notag\\
			&\qquad-\frac1P\sum_{p=1}^P\widetilde d_{n,i}^p
			\widetilde{\mathcal L}_{n,i}^p
			\bigl(D\phi(t_n,x_b^p),\phi(t_n,x_b^p),a,b\bigr)
			\notag\\
			&\qquad+\frac1P\sum_{p=1}^P(\widetilde d_{n,i}^p)^2
			\widetilde{\mathcal R}_{n,i}^p
			\bigl(D\phi(t_n,x_b^p),\phi(t_n,x_b^p),a,b\bigr)
			+O(\Delta t^{3/2}),
			\label{eq:cons-averaged}
		\end{align}
		where the common arguments $(a,b)$ in the branch quantities have only been
		suppressed to shorten the display.
		
		\medskip
		\noindent\textbf{Step 3: Add the ordinary-time discount and the
			$\mathbb P_1$ interpolation error.}
		The interpolation estimate gives, uniformly in the branches,
		\[
		I_h[\phi(t_{n+1},\cdot)|_{\cG_h}](\widetilde X_{n,i}^p)
		=\phi(t_{n+1},\widetilde X_{n,i}^p)+O(h^2),
		\]
		and
		\[
		\frac1{1+r\Delta t}=1-r\Delta t+O(\Delta t^2).
		\]
		Substitution of~\eqref{eq:cons-averaged} into
		\eqref{eq:fixed-operator}, followed by addition of $f\Delta t$ and
		subtraction of $\phi_{n,i}$, gives~\eqref{eq:reflected-expansion-fixed}.
		Indeed, the terms containing $\partial_t\phi$, $\mu$, $\sigma$, $r$, and
		$f$ combine exactly into $-\Delta t[-\partial_t\phi+\mathcal H]$.
		The product of $r\Delta t$ with the boundary sum is
		$O(\Delta t^{3/2})$ because
		$\widetilde d_{n,i}^p=O(\sqrt{\Delta t})$.
		All remainders are uniform on the compact control sets.  Taking the infimum
		over $(a,b)$ and using
		\(
		|\inf F_{a,b}-\inf G_{a,b}|\le\sup|F_{a,b}-G_{a,b}|
		\)
		and $\inf(-F)=-\sup F$ proves
		\eqref{eq:reflected-expansion-controlled}.
		
		\medskip
		\noindent\textbf{Step 4: Boundary nodes.}
		Fix $i\in\mathcal I_h^\partial$ and $b\in B$, and set
		\[
		y_{i,b}=x_i-\ell_h\gamma(x_i,b),\qquad
		\phi_n=\phi(t_n,\cdot)|_{\cG_h}.
		\]
		The $\mathbb P_1$ interpolation estimate and Taylor's formula at $x_i$
		give
		\[
		I_h[\phi_n](y_{i,b})
		=\phi_{n,i}-\ell_h\gamma(x_i,b)\cdot D\phi_{n,i}
		+O(\ell_h^2+h^2).
		\]
		Substitution into~\eqref{eq:fixed-boundary-map} yields, uniformly in $b$,
		\begin{align}
			\cS_{n,i}^{\partial,b}[\phi_n]-\phi_{n,i}
			&=-\frac{\ell_h}{1+\ell_h k(t_n,x_i,b)}
			\mathcal L(t_n,x_i,D\phi_{n,i},\phi_{n,i},b)
			+O(\ell_h^2+h^2)
			\notag\\
			&=-\ell_h\mathcal L(t_n,x_i,D\phi_{n,i},\phi_{n,i},b)+O(h^2),
			\label{eq:boundary-fixed-consistency}
		\end{align}
		where the second equality uses $\ell_h\asymp h$ and boundedness of $k$ and
		$\mathcal L$.  Taking the infimum over $b$ and using uniformity of the
		remainder proves~\eqref{eq:boundary-controlled-consistency}.
	\end{proof}
	
	\subsection{Stability of the parabolic scheme}
	
	\begin{proposition}[Parabolic stability]\label{prop:stability}
		Under Assumptions~\ref{ass:H1}--\ref{ass:H4}, assume
		\begin{equation}
			\Delta t,h\to0,
			\qquad \frac{h^2}{\Delta t}\to0.
			\label{eq:admissible-refinement}
		\end{equation}
		There is a constant $C$, independent of the discretization parameters, such
		that
		\[
		\max_{0\le n\le N_T}\|U^n\|_\infty\le C
		\]
		for all sufficiently fine grids.
	\end{proposition}
	\begin{proof}
		Set
		\[
		F_0=\|f\|_\infty,\qquad G_0=\|g\|_\infty,\qquad A_0=G_0+1.
		\]
		
		\medskip
		\noindent\textbf{Step 1: Auxiliary function.}
		By uniform obliqueness,
		\[
		\nu:=\min_{(x,b)\in\partial\cO\times B}n(x)\cdot\gamma(x,b)>0.
		\]
		Let $d_\partial(x)=\dist(x,\partial\cO)$ in $\cO$ and write
		$L_\rho:=\{x\in\overline\cO:d_\partial(x)<\rho\}$.  Choose a tubular
		radius $\eta>0$ on which $d_\partial$ is $C^3$ and
		$Dd_\partial=-n$ on $\partial\cO$, put $\kappa=1/\nu$, and choose
		$\varepsilon_0\le\min\{\eta/2,1/(2\kappa)\}$.  On $L_{2\varepsilon_0}$ set
		$\chi_0=1-\kappa d_\partial$.  Let $\zeta$ be a smooth cutoff equal to one
		on $L_{\varepsilon_0}$ and zero outside $L_{2\varepsilon_0}$, and extend
		$\widetilde\chi:=\zeta\chi_0$ by zero to the rest of $\overline\cO$.
		Set $\chi:=\widetilde\chi+\|\widetilde\chi^-\|_\infty$.  Then
		\begin{equation}
			\chi\in C^3(\overline\cO),\qquad \chi\ge0,\qquad
			\gamma(x,b)\cdot D\chi(x)\ge1
			\quad (x,b)\in\partial\cO\times B.
			\label{eq:chi}
		\end{equation}
		The cutoff can be chosen flat where it vanishes; hence $\chi$ has a bounded
		$C^3$ extension to a fixed neighbourhood of $\overline\cO$.
		
		\medskip
		\noindent\textbf{Step 2: Upper barrier.}
		For a constant $M_0>0$ to be fixed below, define
		\[
		\psi(t,x)=e^{T-t}(M_0+A_0\chi(x)).
		\]
		Set
		\[
		C_1:=\sup_{(t,x,a)\in[0,T]\times\overline\cO\times A}
		\left|\frac12\Tr(\sigma\sigma^\top D^2\chi)
		+\mu\cdot D\chi\right|<\infty.
		\]
		Since $r,k\ge0$ and $e^{T-t}\ge1$, direct substitution gives, uniformly
		in $a\in A$ and $b\in B$,
		\begin{align}
			[-\partial_t\psi+\mathcal H]
			(t,x,D\psi,D^2\psi,\psi,a)
			&\ge e^{T-t}M_0-e^TA_0C_1-F_0,
			\label{eq:H-psi}\\
			\mathcal L(t,x,D\psi,\psi,b)&\ge A_0-G_0=1
			\qquad (x\in\partial\cO),
			\label{eq:L-psi}
		\end{align}
		Choose
		\begin{equation}
			M_0\ge\max\{e^TA_0C_1+F_0+4,\ \|\Psi\|_\infty\}.
			\label{eq:M0}
		\end{equation}
		Then the two residuals in \eqref{eq:H-psi}--\eqref{eq:L-psi} are bounded
		below by $4$ and $1$, respectively.
		
		\medskip
		\noindent\textbf{Step 3: Consistency estimate for the barriers.}
		Apply Proposition~\ref{prop:nodewise-consistency} to the two fixed smooth
		functions $\psi$ and $-\psi$, and let $C_{\rm cons}$ be a common constant
		for their interior and boundary remainders.  By
		\eqref{eq:admissible-refinement} and $\ell_h\ge c_0h$, for all sufficiently
		fine grids,
		\begin{equation}
			C_{\rm cons}(\Delta t^{3/2}+h^2)\le\Delta t,
			\qquad C_{\rm cons}h^2\le\frac12\ell_h.
			\label{eq:barrier-remainder}
		\end{equation}
		
		On an exiting branch,
		\eqref{eq:R-equals-kL} gives
		\[
		-\widetilde d\,\widetilde{\mathcal L}
		+\widetilde d^2\widetilde{\mathcal R}
		=-\widetilde d\left(1-\frac12\widetilde k\widetilde d\right)
		\widetilde{\mathcal L}.
		\]
		The step bound~\eqref{eq:step-bound} implies
		$\widetilde d=O(\sqrt{\Delta t})$ uniformly.  Hence, for sufficiently small
		$\Delta t$,
		\(
		1-\widetilde k\widetilde d/2\ge1/2
		\)
		on every branch.  Part~\textup{(I)} of
		Proposition~\ref{prop:nodewise-consistency} can therefore be used with a
		fixed sign for every reflected boundary contribution.
		
		\medskip
		\noindent\textbf{Step 4: The upper barrier is a discrete supersolution.}
		Fix $i\in\mathcal I_h^\circ$ and $(a,b)\in A\times B$.  Part~\textup{(I)}
		of Proposition~\ref{prop:nodewise-consistency} gives
		\begin{align}
			&\cS_{n,i}^{a,b}[\psi(t_{n+1},\cdot)]-\psi(t_n,x_i)
			\notag\\
			&\quad=-\Delta t\,
			[-\partial_t\psi+\mathcal H](t_n,x_i,D\psi,D^2\psi,\psi,a)
			\notag\\
			&\qquad\quad+\frac1P\sum_{p=1}^P\left\{
			-\widetilde d_{n,i}^p\,
			\widetilde{\mathcal L}_{n,i}^p[\psi](a,b)
			+(\widetilde d_{n,i}^p)^2
			\widetilde{\mathcal R}_{n,i}^p[\psi](a,b)\right\}
			+\rho_{n,i}^{a,b},
			\label{eq:upper-fixed-expansion}
		\end{align}
		Here $\widetilde{\mathcal L}_{n,i}^p[\psi](a,b)$ and
		$\widetilde{\mathcal R}_{n,i}^p[\psi](a,b)$ denote the two truncated
		operators in Proposition~\ref{prop:nodewise-consistency} evaluated at
		$\psi$; they are zero on a non-exiting branch.  Moreover,
		$|\rho_{n,i}^{a,b}|\le C_{\rm cons}(\Delta t^{3/2}+h^2)\le\Delta t$.
		The first term on the right of~\eqref{eq:upper-fixed-expansion} is at most
		$-4\Delta t$.  A non-exiting branch has $\widetilde d_{n,i}^p=0$.
		On an exiting branch, \eqref{eq:R-equals-kL}, \eqref{eq:L-psi}, and
		$1-\widetilde k_{n,i}^p\widetilde d_{n,i}^p/2\ge1/2$ give
		\[
		-\widetilde d_{n,i}^p\widetilde{\mathcal L}_{n,i}^p[\psi]
		+(\widetilde d_{n,i}^p)^2\widetilde{\mathcal R}_{n,i}^p[\psi]
		=-\widetilde d_{n,i}^p
		\left(1-\frac12\widetilde k_{n,i}^p\widetilde d_{n,i}^p\right)
		\widetilde{\mathcal L}_{n,i}^p[\psi]
		\le-\frac12\widetilde d_{n,i}^p\le0.
		\]
		Consequently, for every fixed $(a,b)$,
		\[
		\cS_{n,i}^{a,b}[\psi(t_{n+1},\cdot)]-\psi(t_n,x_i)
		\le-4\Delta t-\frac1{2P}\sum_{p:\,\mathrm{exit}}
		\widetilde d_{n,i}^p+\Delta t
		\le-3\Delta t.
		\]
		Taking the infimum over $(a,b)$ proves
		\begin{equation}
			S_{n,i}[\psi(t_{n+1},\cdot)]-\psi(t_n,x_i)\le-3\Delta t<0.
			\label{eq:S-psi-interior}
		\end{equation}
		
		Now fix $i\in\mathcal I_h^\partial$ and $b\in B$.  The fixed-control
		boundary expansion~\eqref{eq:boundary-fixed-consistency} gives
		\[
		\cS_{n,i}^{\partial,b}[\psi(t_n,\cdot)]-\psi(t_n,x_i)
		=-\ell_h\mathcal L(t_n,x_i,D\psi,\psi,b)
		+\rho_{n,i}^{\partial,b},
		\qquad |\rho_{n,i}^{\partial,b}|\le C_{\rm cons}h^2.
		\]
		Since \eqref{eq:L-psi} holds for every $b$,
		\[
		\cS_{n,i}^{\partial,b}[\psi(t_n,\cdot)]-\psi(t_n,x_i)
		\le-\ell_h+C_{\rm cons}h^2\le-\frac12\ell_h.
		\]
		Taking the infimum over $b$ proves
		\begin{equation}
			S_{n,i}^{\partial}[\psi(t_n,\cdot)]-\psi(t_n,x_i)
			\le-\frac12\ell_h<0.
			\label{eq:S-psi-boundary}
		\end{equation}
		Thus \eqref{eq:S-psi-interior} and \eqref{eq:S-psi-boundary} verify the
		upper barrier separately on the interior and boundary node sets.  Notice
		that the boundary estimate uses only $\ell_h\asymp h$ and no time-step
		condition.
		
		\medskip
		\noindent\textbf{Step 5: Lower barrier.}
		Define
		\[
		\underline\psi(t,x):=-e^{T-t}(M_0+A_0\chi(x))=-\psi(t,x).
		\]
		Direct substitution, exactly as in Step~2, gives
		\begin{align}
			[-\partial_t\underline\psi+\mathcal H]
			(t,x,D\underline\psi,D^2\underline\psi,\underline\psi,a)&\le-4,
			\label{eq:H-lower-psi}\\
			\mathcal L(t,x,D\underline\psi,\underline\psi,b)&\le-1
			\qquad(x\in\partial\cO).
			\label{eq:L-lower-psi}
		\end{align}
		Fix an interior node and $(a,b)$.  Substitution of $\underline\psi$ into
		the fixed-control expansion~\eqref{eq:upper-fixed-expansion}, with a new
		remainder satisfying the same absolute bound, gives a volume contribution
		at least $4\Delta t$.  On every exiting branch,
		\[
		-\widetilde d\left(1-\frac12\widetilde k\widetilde d\right)
		\widetilde{\mathcal L}[\underline\psi]
		\ge\frac12\widetilde d,
		\]
		because $\widetilde{\mathcal L}[\underline\psi]\le-1$.  Hence, for every
		fixed $(a,b)$,
		\[
		\cS_{n,i}^{a,b}[\underline\psi(t_{n+1},\cdot)]
		-\underline\psi(t_n,x_i)
		\ge4\Delta t+\frac1{2P}\sum_{p:\,\mathrm{exit}}
		\widetilde d_{n,i}^p-\Delta t
		\ge3\Delta t.
		\]
		Because this lower bound holds for every control pair, taking the infimum
		preserves it and yields
		\begin{equation}
			S_{n,i}[\underline\psi(t_{n+1},\cdot)]
			-\underline\psi(t_n,x_i)\ge3\Delta t>0,
			\qquad i\in\mathcal I_h^\circ.
			\label{eq:S-lower-psi-interior}
		\end{equation}
		At a boundary node, the fixed-control boundary expansion and
		\eqref{eq:L-lower-psi} give, for every $b$,
		\[
		\cS_{n,i}^{\partial,b}[\underline\psi(t_n,\cdot)]
		-\underline\psi(t_n,x_i)
		\ge\ell_h-C_{\rm cons}h^2\ge\frac12\ell_h.
		\]
		Taking the infimum over $b$ therefore gives
		\begin{equation}
			S_{n,i}^{\partial}[\underline\psi(t_n,\cdot)]
			-\underline\psi(t_n,x_i)
			\ge\frac12\ell_h>0.
			\label{eq:S-lower-psi-boundary}
		\end{equation}
		
		\medskip
		\noindent\textbf{Step 6: Comparison and the $L^\infty$ estimate.}
		At $t=T$, the choice of $M_0$ gives
		\[
		\underline\psi(T,x_i)\le U_i^{N_T}\le\psi(T,x_i)
		\qquad(i\in\mathcal I_h).
		\]
		Assume this order at level $n+1$.  Interior monotonicity together with
		\eqref{eq:S-psi-interior}--\eqref{eq:S-lower-psi-interior} first yields
		\[
		\underline\psi(t_n,x_i)\le U_i^n\le\psi(t_n,x_i)
		\qquad(i\in\mathcal I_h^\circ).
		\]
		For the same time layer, write $T_n^q$ for the boundary map with interior
		vector $q$.  It is order preserving in both arguments and is a contraction.
		The just-proved interior bounds and
		\eqref{eq:S-psi-boundary}--\eqref{eq:S-lower-psi-boundary} give
		\[
		T_n^{U_\circ^n}(\psi_\partial^n)
		\le T_n^{\psi_\circ^n}(\psi_\partial^n)\le\psi_\partial^n,
		\qquad
		T_n^{U_\circ^n}(\underline\psi_\partial^n)
		\ge T_n^{\underline\psi_\circ^n}(\underline\psi_\partial^n)
		\ge\underline\psi_\partial^n.
		\]
		Iterating the monotone contraction $T_n^{U_\circ^n}$ from the two barriers
		therefore brackets its unique fixed point between them.  Hence the same bound
		holds on $\mathcal I_h^\partial$.  Backward induction proves
		\[
		\max_{0\le n\le N_T}\|U^n\|_\infty
		\le e^T(M_0+A_0\|\chi\|_\infty),
		\]
		which is independent of $h$ and $\Delta t$.
	\end{proof}
	\begin{remark}\upshape
		Condition~\eqref{eq:admissible-refinement} is used in the barrier proof to
		absorb the $O(h^2)$ consistency remainder into the $O(\Delta t)$ volume
		margin.  It is a refinement condition for the present stability and
		convergence results, not a CFL restriction for positivity or monotonicity;
		we do not claim the stability bound for arbitrary independent choices of
		$h$ and $\Delta t$.
	\end{remark}
	\section{Convergence to the viscosity solution}
	\label{sec:convergence}
	
	For $t\in[t_n,t_{n+1})$ set
	\begin{equation}
		u^{h,\dt}(t,x)=I_h[U^n](x),
		\qquad
		u^{h,\dt}(T,x)=I_h[U^{N_T}](x).
		\label{eq:extension}
	\end{equation}
	All limits in this section are taken along families satisfying~\eqref{eq:admissible-refinement}.  Define the nodal half-relaxed limits
	\begin{align}
		\overline u(t,x)
		&=\limsup_{\substack{h,\dt\to0,\ h^2/\dt\to0\\
				(t_n,x_i)\to(t,x)}}U_i^n,
		\label{eq:upper-relaxed}\\
		\underline u(t,x)
		&=\liminf_{\substack{h,\dt\to0,\ h^2/\dt\to0\\
				(t_n,x_i)\to(t,x)}}U_i^n.
		\label{eq:lower-relaxed}
	\end{align}
	They are also the half-relaxed limits of~\eqref{eq:extension}: every interpolated value is a convex combination of nodal values at distance $O(h)$, whereas every nodal value is itself an interpolated value.  Proposition~\ref{prop:stability} makes both limits finite.
	
	\begin{theorem}[Uniform viscosity convergence]\label{thm:convergence}
		Under Assumptions~\ref{ass:H1}--\ref{ass:H4}, let~\eqref{eq:admissible-refinement} hold.  Then
		$u^{h,\dt}$ converges uniformly on $\overline\cO_T$ to the unique viscosity solution of~\eqref{eq:HJB-main}.
	\end{theorem}
	
	\begin{proof}
		By the smooth-test reduction~\cite[Sections~2--3]{CrandallIshiiLions1992},
		it is enough to use test functions having a bounded $C^3$ extension to a
		fixed neighbourhood of the contact point.  Multiplication by a cutoff equal
		to one near that point does not change the argument: the interior stencil has
		diameter $O(\sqrt{\dt}+h)$ by~\eqref{eq:step-bound}, and the same-level
		boundary stencil has diameter $O(h)$.
		
		\medskip
		\noindent\textbf{Step 1: Strict maximum and recovery sequence.}
		Let $z_0=(t_0,x_0)\in\overline\cO_T$, and suppose that
		$\overline u-\phi$ has a local maximum at $z_0$.  Subtracting a constant,
		assume $(\overline u-\phi)(z_0)=0$.  For fixed $\varepsilon>0$, set
		\begin{equation}
			\phi^\varepsilon(t,x)
			:=\phi(t,x)+\varepsilon\bigl(|t-t_0|^2+|x-x_0|^4\bigr).
			\label{eq:strict-upper-test}
		\end{equation}
		Then $\overline u-\phi^\varepsilon$ has a strict local maximum $0$ at
		$z_0$, and the derivatives of $\phi^\varepsilon$ used in the equation agree
		with those of $\phi$ at $z_0$.
		
		Choose compact relative neighbourhoods
		$K'\Subset K\subset\overline\cO_T$ such that the strict maximum lies in
		$K'$ and $\overline u-\phi^\varepsilon<0$ on $K\setminus K'$.  On a
		subsequence of discretizations realizing the upper relaxed limit, let
		$\mathcal G_j^{\rm st}$ be the corresponding space--time nodal grid and
		write $U_j(t_n,x_i):=U_i^n$.  Set
		\begin{equation}
			c_j:=\max_{(t_n,x_i)\in K\cap\mathcal G_j^{\rm st}}
			\{U_i^n-\phi^\varepsilon(t_n,x_i)\},
			\qquad
			z_j=(t_{n_j},x_{i_j})
			\label{eq:upper-grid-maximum}
		\end{equation}
		at a maximizing node.  A recovery sequence
		$y_j\in K\cap\mathcal G_j^{\rm st}$ with
		$y_j\to z_0$ and $U_j(y_j)\to\overline u(z_0)$ gives
		\[
		\liminf_{j\to\infty}c_j
		\ge\lim_{j\to\infty}\bigl(U_j(y_j)-\phi^\varepsilon(y_j)\bigr)=0.
		\]
		By stability, every subsequence of $(z_j,c_j)$ has a further subsequence
		with $z_j\to z_*\in K$ and $c_j\to c\ge0$.  Since
		$U_{i_j}^{n_j}=\phi^\varepsilon(z_j)+c_j$, the definition of
		$\overline u$ gives
		\[
		(\overline u-\phi^\varepsilon)(z_*)\ge c\ge0.
		\]
		Strictness forces $z_*=z_0$ and $c=0$.  This argument applies to every
		subsequence; hence
		\begin{equation}
			z_j\to z_0,\qquad c_j\to0.
			\label{eq:contact-limit}
		\end{equation}
		For large $j$, $z_j\in K'$ and every stencil at $z_j$ is contained in
		$K$.  Therefore the maximizing property yields
		\begin{equation}
			U^{n_j+1}\le
			\phi^\varepsilon(t_{n_j+1},\cdot)|_{\cG_{h_j}}+c_j
			\quad\hbox{on every next-time stencil},
			\label{eq:max-next-stencil}
		\end{equation}
		whenever $n_j<N_{T,j}$, and
		\begin{equation}
			U^{n_j}\le
			\phi^\varepsilon(t_{n_j},\cdot)|_{\cG_{h_j}}+c_j
			\quad\hbox{on every same-time boundary stencil}.
			\label{eq:max-boundary-stencil}
		\end{equation}
		
		\medskip
		\noindent\textbf{Step 2: Monotonicity and signed constant shifts.}
		The constants $c_j$ need not have a fixed sign.  For an interior
		fixed-control operator, define
		\[
		\delta_{j}^{a,b}(c):=
		\cS_{n_j,i_j}^{a,b}[V+c]-\cS_{n_j,i_j}^{a,b}[V]-c.
		\]
		By~\eqref{eq:shift-defect}, uniformly in the controls,
		\begin{equation}
			|\delta_j^{a,b}(c)|
			\le C|c|\left(\dt_j+\frac1P\sum_{p=1}^P
			\widetilde d_{n_j,i_j}^p(a,b)\right).
			\label{eq:interior-signed-shift}
		\end{equation}
		For a boundary fixed-control row the formula is exact:
		\begin{equation}
			\cS_{n_j,i_j}^{\partial,b}[V+c]
			-\cS_{n_j,i_j}^{\partial,b}[V]-c
			=-\frac{\ell_{h_j}k(t_{n_j},x_{i_j},b)}
			{1+\ell_{h_j}k(t_{n_j},x_{i_j},b)}\,c.
			\label{eq:boundary-signed-shift}
		\end{equation}
		Thus, after division by the natural interior scale
		$\dt_j+P^{-1}\sum_p\widetilde d^p$ or by the boundary scale $\ell_{h_j}$,
		the shift errors produced by $c_j$ are $O(|c_j|)$.
		
		If $i_j\in\mathcal I_{h_j}^\circ$ and $n_j<N_{T,j}$, the scheme,
		\eqref{eq:max-next-stencil}, and monotonicity imply, for every fixed
		$(a,b)$,
		\begin{equation}
			0\le
			\cS_{n_j,i_j}^{a,b}[\phi^\varepsilon(t_{n_j+1},\cdot)]
			-\phi^\varepsilon(z_j)+\delta_j^{a,b}(c_j).
			\label{eq:upper-interior-discrete}
		\end{equation}
		If $i_j\in\mathcal I_{h_j}^\partial$, then for every fixed $b$,
		\begin{equation}
			0\le
			\cS_{n_j,i_j}^{\partial,b}[\phi^\varepsilon(t_{n_j},\cdot)]
			-\phi^\varepsilon(z_j)
			-\frac{\ell_{h_j}k(t_{n_j},x_{i_j},b)}
			{1+\ell_{h_j}k(t_{n_j},x_{i_j},b)}c_j.
			\label{eq:upper-boundary-discrete}
		\end{equation}
		Indeed, at an interior node
		$U_{i_j}^{n_j}=S_{n_j,i_j}[U^{n_j+1}]$, whereas at a boundary node
		$U_{i_j}^{n_j}=S_{n_j,i_j}^{\partial}[U^{n_j}]$; these are precisely the
		two lines of~\eqref{eq:scheme}.
		
		\medskip
		\noindent\textbf{Step 3: Interior subsolution condition.}
		Assume $t_0<T$ and $x_0\in\cO$.  For all sufficiently large $j$,
		$i_j\in\mathcal I_{h_j}^\circ$ and every branch stays in $\cO$, uniformly
		in $(a,b)$.  Fix arbitrary $(a,b)\in A\times B$.  In
		\eqref{eq:upper-interior-discrete} all $\widetilde d^p$ vanish.  Part
		\textup{(I)} of Proposition~\ref{prop:nodewise-consistency}, divided by
		$\dt_j$, therefore gives
		\[
		[-\partial_t\phi^\varepsilon+\mathcal H]
		(t_{n_j},x_{i_j},D\phi^\varepsilon,D^2\phi^\varepsilon,
		\phi^\varepsilon,a)
		\le C\left(\sqrt{\dt_j}+\frac{h_j^2}{\dt_j}+|c_j|\right).
		\]
		The right-hand side tends to zero by
		\eqref{eq:admissible-refinement} and~\eqref{eq:contact-limit}.  Passing to
		the limit and using the arbitrariness of $a$ gives
		\[
		-\partial_t\phi^\varepsilon(z_0)
		+H(t_0,x_0,D\phi^\varepsilon(z_0),D^2\phi^\varepsilon(z_0),
		\overline u(z_0))\le0.
		\]
		Letting $\varepsilon\downarrow0$ proves~\eqref{eq:sub-interior}.
		
		\medskip
		\noindent\textbf{Step 4: Boundary subsolution condition.}
		Assume $t_0<T$ and $x_0\in\partial\cO$.  Passing to a subsequence, the
		contact nodes are either all boundary nodes or all interior nodes.
		
		\smallskip
		\noindent\textit{Case A: boundary contact nodes.}
		Let $i_j\in\mathcal I_{h_j}^\partial$ and fix $b\in B$.  Divide
		\eqref{eq:upper-boundary-discrete} by $\ell_{h_j}$.  Part~\textup{(II)} of
		Proposition~\ref{prop:nodewise-consistency},
		\eqref{eq:boundary-signed-shift}, $c_j\to0$, and
		$h_j^2/\ell_{h_j}=O(h_j)$ yield
		\[
		\mathcal L(t_0,x_0,D\phi^\varepsilon(z_0),
		\overline u(z_0),b)\le0.
		\]
		This holds for every $b$, and hence
		\begin{equation}
			L(t_0,x_0,D\phi^\varepsilon(z_0),\overline u(z_0))\le0.
			\label{eq:boundary-node-sub}
		\end{equation}
		
		\smallskip
		\noindent\textit{Case B: interior contact nodes.}
		Let $i_j\in\mathcal I_{h_j}^\circ$, fix $(a,b)\in A\times B$, and abbreviate
		\[
		\begin{aligned}
			d_{j,p}&:=\widetilde d_{n_j,i_j}^p(a,b),&
			k_{j,p}&:=\widetilde k_{n_j,i_j}^p(a,b),&
			x_{j,p}^b&:=x_{n_j,i_j}^{b,p}(a),\\
			s_j&:=\dt_j+\frac1P\sum_{p=1}^P d_{j,p},&
			\lambda_j&:=\frac{\dt_j}{s_j}.&&
		\end{aligned}
		\]
		For an exiting branch set
		\[
		L_{j,p}^b:=\mathcal L(t_{n_j},x_{j,p}^b,
		D\phi^\varepsilon(t_{n_j},x_{j,p}^b),
		\phi^\varepsilon(t_{n_j},x_{j,p}^b),b).
		\]
		The value assigned to $L_{j,p}^b$ on a non-exiting branch is irrelevant
		because then $d_{j,p}=0$.  Using
		$\mathcal R=(k/2)\mathcal L$ in Part~\textup{(I)} of Proposition
		~\ref{prop:nodewise-consistency}, inequality
		\eqref{eq:upper-interior-discrete} becomes
		\begin{equation}
			\lambda_j A_j^a+\frac1{Ps_j}\sum_{p=1}^P d_{j,p}
			\left(1-\frac12k_{j,p}d_{j,p}\right)L_{j,p}^b
			\le\eta_j,
			\label{eq:sub-normalized}
		\end{equation}
		where
		\begin{align*}
			A_j^a&:=[-\partial_t\phi^\varepsilon+\mathcal H]
			(t_{n_j},x_{i_j},D\phi^\varepsilon(z_j),D^2\phi^\varepsilon(z_j),
			\phi^\varepsilon(z_j),a),\\
			|\eta_j|&\le C\left(
			\sqrt{\dt_j}+\frac{h_j^2}{\dt_j}+|c_j|\right)\longrightarrow0.
		\end{align*}
		Here $s_j\ge\dt_j$ controls the consistency remainder, and
		\eqref{eq:interior-signed-shift} controls the last term by $C|c_j|$.
		
		This includes the case in which no branch exits: then every $d_{j,p}=0$,
		$s_j=\dt_j$, and $\lambda_j=1$.  Extract a subsequence such that
		$\lambda_j\to\lambda\in[0,1]$.  Since
		$\max_p d_{j,p}=O(\sqrt{\dt_j})$,
		\begin{equation}
			\frac1{Ps_j}\sum_{p=1}^P d_{j,p}
			\left(1-\frac12k_{j,p}d_{j,p}\right)
			=1-\lambda_j+o(1).
			\label{eq:boundary-weight-limit}
		\end{equation}
		Every exiting projection $x_{j,p}^b$ tends to $x_0$.  Since $k$ is bounded
		and $\max_p d_{j,p}\to0$, the factors
		$1-k_{j,p}d_{j,p}/2$ are nonnegative for all large $j$; hence the absolute
		sum of the weights in~\eqref{eq:boundary-weight-limit} is uniformly
		bounded.  Uniform continuity of $\mathcal L$ therefore gives
		\[
		\frac1{Ps_j}\sum_{p=1}^P d_{j,p}
		\left(1-\frac12k_{j,p}d_{j,p}\right)
		\bigl[L_{j,p}^b-\mathcal L(t_0,x_0,D\phi^\varepsilon(z_0),
		\overline u(z_0),b)\bigr]\longrightarrow0.
		\]
		Passing to the limit in~\eqref{eq:sub-normalized} yields
		\begin{equation}
			\lambda[-\partial_t\phi^\varepsilon+\mathcal H]
			(t_0,x_0,D\phi^\varepsilon(z_0),D^2\phi^\varepsilon(z_0),
			\overline u(z_0),a)
			+(1-\lambda)\mathcal L(t_0,x_0,D\phi^\varepsilon(z_0),
			\overline u(z_0),b)\le0.
			\label{eq:boundary-convex-sub}
		\end{equation}
		If both $-\partial_t\phi^\varepsilon+H$ and $L$ were positive at $z_0$,
		compactness of $A$ and $B$ and continuity of the control-resolved operators
		would give fixed controls $a_0\in A$, $b_0\in B$, and $\delta>0$ such that
		both corresponding quantities are at least $\delta$.  Repeating the
		preceding subsequence extraction for $(a_0,b_0)$, the left-hand side of
		\eqref{eq:boundary-convex-sub} would be at least
		$\lambda\delta+(1-\lambda)\delta=\delta>0$, a contradiction.  Thus
		\[
		\min\{-\partial_t\phi^\varepsilon+H,L\}(z_0)\le0.
		\]
		Together with Case~A, and then letting $\varepsilon\downarrow0$, this proves
		\eqref{eq:sub-boundary}.  If boundary and interior contact nodes alternate,
		one takes an infinite subsequence of either type; there is no third case.
		
		\medskip
		\noindent\textbf{Step 5: Terminal subsolution condition.}
		Let $t_0=T$.  Repeat Step~1 with the strict test
		$\phi^\varepsilon$ and obtain $z_j=(t_{n_j},x_{i_j})\to(T,x_0)$ and
		$c_j\to0$.
		
		If $t_{n_j}=T$ along an infinite subsequence, then
		\[
		U_{i_j}^{N_{T,j}}=\Psi(x_{i_j}),\qquad
		U_{i_j}^{N_{T,j}}=\phi^\varepsilon(T,x_{i_j})+c_j.
		\]
		Passing to the limit gives
		$\overline u(T,x_0)=\phi(T,x_0)=\Psi(x_0)$, so the terminal-data member of
		the relaxed minimum is zero.
		
		Otherwise $t_{n_j}<T$ for all sufficiently large $j$.  The next time level
		may be $T$, but Steps~2--4 remain valid.  Letting $\varepsilon\downarrow0$,
		if $x_0\in\cO$ they give
		$-\phi_t+H\le0$; if $x_0\in\partial\cO$ they give
		$\min\{-\phi_t+H,L\}\le0$.  Consequently,
		\[
		\min\{-\phi_t+H,\overline u-\Psi\}(T,x_0)\le0
		\quad(x_0\in\cO),
		\]
		and
		\[
		\min\{-\phi_t+H,L,\overline u-\Psi\}(T,x_0)\le0
		\quad(x_0\in\partial\cO).
		\]
		Thus $\overline u$ is a viscosity subsolution.
		
		\medskip
		\noindent\textbf{Step 6: Supersolution inequalities.}
		Let $\underline u-\phi$ have a local minimum $0$ at
		$z_0=(t_0,x_0)$ and set
		\[
		\phi_\varepsilon(t,x)
		:=\phi(t,x)-\varepsilon\bigl(|t-t_0|^2+|x-x_0|^4\bigr).
		\]
		For brevity, denote the control-resolved residuals at $z_0$ by
		\begin{align*}
			A_\varepsilon^a
			&:=[-\partial_t\phi_\varepsilon+\mathcal H]
			(t_0,x_0,D\phi_\varepsilon(z_0),D^2\phi_\varepsilon(z_0),
			\underline u(z_0),a),\\
			L_\varepsilon^b
			&:=\mathcal L(t_0,x_0,D\phi_\varepsilon(z_0),
			\underline u(z_0),b).
		\end{align*}
		Repeating Step~1 with minima in place of maxima gives recovery nodes
		$z_j\to z_0$, constants $c_j\to0$, and the corresponding lower stencil
		comparisons.  Since both discrete updates are infima, these comparisons are
		applied below to controls chosen within $\dt_j^2$ (respectively
		$\ell_{h_j}^2$) of the relevant infimum.
		For $t_0<T$ and $x_0\in\cO$, choose controls within $\dt_j^2$ of the
		interior infimum and repeat Step~3.  After division by $\dt_j$, the remainder
		is $O(\sqrt{\dt_j}+h_j^2/\dt_j+|c_j|+\dt_j)=o(1)$, and compactness gives a
		control $a_*$ for which $A_\varepsilon^{a_*}\ge0$; hence
		$-\partial_t\phi_\varepsilon+H\ge0$.
		
		Let now $t_0<T$ and $x_0\in\partial\cO$.  A subsequence of boundary
		contact nodes is treated by Case~A with an $\ell_{h_j}^2$-minimizer of the
		boundary infimum, and gives $L_\varepsilon^{b_*}\ge0$, hence $L\ge0$.
		For interior contact nodes approaching the boundary, choose $(a_j,b_j)$
		so that the fixed-control value is within $\dt_j^2$ of the interior
		infimum, and use the same
		$s_j=\dt_j+P^{-1}\sum_p d_{j,p}$ and $\lambda_j=\dt_j/s_j$ as in Case~B.
		Since $s_j\ge\dt_j$, the normalized selection error satisfies
		$\dt_j^2/s_j\le\dt_j\to0$.  After extraction,
		$(a_j,b_j)\to(a_*,b_*)$ and $\lambda_j\to\lambda\in[0,1]$.
		The estimates in Case~B are uniform in the controls; with $d_{j,p}=0$ on
		non-exiting branches, they also cover changes in the exiting-branch set.
		The reversed Case~B inequality and~\eqref{eq:boundary-weight-limit} then give
		\[
		\lambda A_\varepsilon^{a_*}+(1-\lambda)L_\varepsilon^{b_*}\ge0,
		\qquad 0\le\lambda\le1.
		\]
		Thus at least one residual is nonnegative and
		$\max\{-\partial_t\phi_\varepsilon+H,L\}(z_0)\ge0$.  Alternating contact
		types are covered by passing to an infinite subsequence.  This implication
		also includes the endpoint cases $\lambda=0$ and $\lambda=1$.
		
		At $t_0=T$, the same construction either reaches the terminal layer and
		gives $\underline u-\Psi\ge0$, or reduces to the preceding nonterminal
		argument.  Letting $\varepsilon\downarrow0$ yields the terminal relaxed
		maximum inequalities in Definition~\ref{def:viscosity}.  Therefore
		$\underline u$ is a viscosity supersolution.
		
		\medskip
		\noindent\textbf{Step 7: Comparison and uniform convergence.}
		By definition, $\underline u\le\overline u$.  Steps~1--6 and
		Proposition~\ref{thm:wellposed} give $\overline u\le\underline u$.
		Hence $u:=\overline u=\underline u$ is continuous and is a viscosity
		solution of~\eqref{eq:HJB-main}; the same comparison principle gives
		uniqueness.  The remainder of the argument proves uniform convergence to
		this $u$.
		
		If convergence were not uniform, compactness of $\overline\cO_T$ would
		give $\varepsilon_0>0$, a subsequence, and points $z_j\to z_*$ with
		$|u^{h_j,\dt_j}(z_j)-u(z_j)|\ge\varepsilon_0$.  A positive discrepancy
		would imply $\overline u(z_*)\ge u(z_*)+\varepsilon_0$; a negative one would
		imply $\underline u(z_*)\le u(z_*)-\varepsilon_0$.  Both contradict the
		equality of the half-relaxed limits with $u$.
	\end{proof}
	
	\begin{theorem}[Smooth-solution error bound]
		\label{thm:rate}
		Under Assumptions~\ref{ass:H1}--\ref{ass:H4}, let
		$u\in C_b^3(\mathcal N_T)$ be a classical solution of
		\eqref{eq:HJB-main} on a neighbourhood
		$\mathcal N_T\supset[0,T]\times\overline\cO$, where $C_b^3$ denotes
		bounded continuous derivatives of total order at most three.  Here
		``classical solution'' means that all three relations in
		\eqref{eq:HJB-main} hold pointwise; in particular,
		$\sup_{b\in B}q_b(t,x)=0$ on $[0,T)\times\partial\cO$, with $q_b$ defined
		below.  Set
		\[
		q_b(t,x):=\mathcal L(t,x,Du(t,x),u(t,x),b),
		\qquad (t,x,b)\in[0,T]\times\partial\cO\times B,
		\]
		and suppose that
		\begin{equation}
			\sup_{t\in[0,T],\,b\in B}
			\|q_b(t,\cdot)\|_{C^{1,1}(\partial\cO)}<\infty.
			\label{eq:residual-C11}
		\end{equation}
		Here $C^{1,1}(\partial\cO)$ is understood intrinsically on the compact
		hypersurface.  Fix a finite $C^2$ atlas
		$\{(U_\alpha,\kappa_\alpha)\}_\alpha$, with
		$\kappa_\alpha:U_\alpha\subset\partial\cO\to V_\alpha\subset\R^{N-1}$,
		and set
		\[
		\|v\|_{C^{1,1}(\partial\cO)}
		:=\max_\alpha
		\|v\circ\kappa_\alpha^{-1}\|_{C^{1,1}(V_\alpha)},
		\]
		where
		\[
		\|w\|_{C^{1,1}(V)}
		:=\|w\|_{L^\infty(V)}+\|Dw\|_{L^\infty(V)}
		+\operatorname{Lip}_V(Dw).
		\]
		Different finite atlases give equivalent norms.  In particular, if
		$v\in C^{1,1}(\partial\cO)$ attains a maximum at $x\in\partial\cO$, then,
		for $y\in\partial\cO$ sufficiently close to $x$,
		\begin{equation}
			v(y)\ge v(x)-C\|v\|_{C^{1,1}(\partial\cO)}|x-y|^2,
			\label{eq:C11-boundary-maximum}
		\end{equation}
		where $C$ depends only on the fixed atlas.
		Then the numerical solution satisfies
		\begin{equation}
			\max_{0\le n\le N_T}\max_{i\in\mathcal I_h}
			|U_i^n-u(t_n,x_i)|
			\le C\left(\dt^{1/2}+\frac{h^2}{\dt}\right).
			\label{eq:full-rate}
		\end{equation}
		The two terms on the right-hand side of~\eqref{eq:full-rate} are balanced
		by $\dt\asymp h^{4/3}$, which yields $O(h^{2/3})$; the commonly used
		choice $\dt\asymp h$ yields $O(h^{1/2})$.
	\end{theorem}
	
	\begin{proof}
		Set
		\[
		u_i^n:=u(t_n,x_i),\qquad
		u_h^n:=(u_j^n)_{j\in\mathcal I_h},\qquad
		e_i^n:=U_i^n-u_i^n,
		\]
		and
		\[
		E_n^\circ:=\max_{i\in\mathcal I_h^\circ}|e_i^n|,\qquad
		E_n^\partial:=\max_{i\in\mathcal I_h^\partial}|e_i^n|,\qquad
		E_n:=\max\{E_n^\circ,E_n^\partial\}.
		\]
		Here and below, $C$ is independent of $h,\dt,n,i,a,$ and $b$; it depends
		only on $T$, the coefficient, domain, and mesh bounds, the norm in
		\eqref{eq:residual-C11}, and $\|u\|_{C_b^3(\mathcal N_T)}$.
		
		\medskip
		\noindent\textbf{Interior nodes.}
		For $i\in\mathcal I_h^\circ$, let $\widehat{\cS}_{n,i}^{a,b}$ be the
		right-hand side of~\eqref{eq:fixed-operator} with the interpolation value
		$I_h[V](\widetilde X_p)$ replaced by the exact value
		$u(t_{n+1},\widetilde X_p)$, and set
		$\widehat S_{n,i}:=\inf_{a\in A,b\in B}\widehat{\cS}_{n,i}^{a,b}$.
		Adding and subtracting $S_{n,i}[u_h^{n+1}]$ and $\widehat S_{n,i}$ gives
		the exact decomposition
		\begin{equation}
			e_i^n=I_{1,i}^n+I_{2,i}^n+I_{3,i}^n,
			\label{eq:three-error-parts}
		\end{equation}
		where
		\begin{align*}
			I_{1,i}^n&:=\widehat S_{n,i}-u_i^n,\\
			I_{2,i}^n&:=S_{n,i}[U^{n+1}]-S_{n,i}[u_h^{n+1}],\\
			I_{3,i}^n&:=S_{n,i}[u_h^{n+1}]-\widehat S_{n,i}.
		\end{align*}
		
		The consistency calculation before the interpolation error is introduced
		gives
		\begin{align}
			I_{1,i}^n
			=-\sup_{a\in A,\,b\in B}\Biggl\{&
			\dt A_{n,i}^a
			+\frac1P\sum_{p=1}^P\widetilde d_{n,i}^p(a,b)
			\left(1-\frac12\widetilde k_{n,i}^p(a,b)
			\widetilde d_{n,i}^p(a,b)\right)
			q_b(t_n,x_b^p)\Biggr\}
			+O(\dt^{3/2}),
			\label{eq:I1-consistency}
		\end{align}
		where
		\[
		A_{n,i}^a:=
		\bigl[-\partial_tu+\mathcal H\bigr]
		(t_n,x_i,Du_i^n,D^2u_i^n,u_i^n,a),
		\]
		and a summand in~\eqref{eq:I1-consistency} is understood to be zero on a
		non-exiting branch.  Since $u$ satisfies the two classical equations,
		\[
		\sup_{a\in A}A_{n,i}^a=0,
		\qquad \sup_{b\in B}q_b(t,x)=0
		\quad (x\in\partial\cO),
		\]
		so every term inside the supremum in~\eqref{eq:I1-consistency} is
		nonpositive for sufficiently small $\dt$.
		
		We now bound that supremum from below.  Choose
		$a_*\in\arg\max_a A_{n,i}^a$.  If no branch for $a_*$ exits, the expression
		in braces vanishes for this $a_*$.  Otherwise, let $\bar x_i$ be a closest
		point of $x_i$ on $\partial\cO$ and choose
		$b_*\in\arg\max_b q_b(t_n,\bar x_i)$.  The step bound
		\eqref{eq:step-bound} first gives
		$\dist(x_i,\partial\cO)\le|X_{n,i}^{*,p_0}(a_*)-x_i|\le C\sqrt\dt$
		for any exiting branch $p_0$.  For every exiting branch generated by
		$(a_*,b_*)$, the definition of the oblique projection,
		Lemma~\ref{lem:oblique-distance}, and the triangle inequality give
		\[
		\begin{aligned}
			|x_{n,i}^{b_*,p}(a_*)-\bar x_i|
			&\le |x_{n,i}^{b_*,p}(a_*)-X_{n,i}^{*,p}(a_*)|
			+|X_{n,i}^{*,p}(a_*)-x_i|+|x_i-\bar x_i|\\
			&=\frac12\widetilde d_{n,i}^p(a_*,b_*)
			+|X_{n,i}^{*,p}(a_*)-x_i|
			+\dist(x_i,\partial\cO)
			\le C\sqrt\dt .
		\end{aligned}
		\]
		Moreover, $q_{b_*}(t_n,\cdot)\le0$ on $\partial\cO$ and attains its maximum
		zero at $\bar x_i$.  Hence its tangential gradient vanishes there, and the
		uniform $C^{1,1}$ bound~\eqref{eq:residual-C11} and
		\eqref{eq:C11-boundary-maximum} give
		\[
		-C|x_{n,i}^{b_*,p}(a_*)-\bar x_i|^2
		\le q_{b_*}(t_n,x_{n,i}^{b_*,p}(a_*))\le0.
		\]
		Since the supremum in~\eqref{eq:I1-consistency} is at least its value at
		the single pair $(a_*,b_*)$, the preceding estimate and
		$\widetilde d_{n,i}^p\le C\sqrt\dt$ prove
		\[
		-C\dt^{3/2}\le
		\sup_{a,b}\Biggl\{\dt A_{n,i}^a
		+\frac1P\sum_{p=1}^P\widetilde d_{n,i}^p
		\left(1-\frac12\widetilde k_{n,i}^p\widetilde d_{n,i}^p\right)
		q_b(t_n,x_b^p)\Biggr\}\le0,
		\]
		and therefore
		\begin{equation}
			|I_{1,i}^n|\le C\dt^{3/2}.
			\label{eq:I1-bound}
		\end{equation}
		
		Next, $|\inf F-\inf G|\le\sup|F-G|$, positivity of $I_h$, and
		$e^{-\widetilde k\widetilde d}/(1+r\dt)\le1$ give
		\begin{align}
			|I_{2,i}^n|
			&\le\sup_{a,b}
			\frac1{1+r(t_n,x_i,a)\dt}\frac1P\sum_{p=1}^P
			e^{-\widetilde k_{n,i}^p\widetilde d_{n,i}^p}
			I_h[|e^{n+1}|](\widetilde X_{n,i}^p)
			\le E_{n+1}.
			\label{eq:I2-bound}
		\end{align}
		The same inequality and the $\mathbb P_1$ interpolation estimate yield
		\begin{align}
			|I_{3,i}^n|
			&\le\sup_{a,b}
			\frac1{1+r(t_n,x_i,a)\dt}\frac1P\sum_{p=1}^P
			e^{-\widetilde k_{n,i}^p\widetilde d_{n,i}^p}
			\left|I_h[u_h^{n+1}](\widetilde X_{n,i}^p)
			-u(t_{n+1},\widetilde X_{n,i}^p)\right|
			\le Ch^2.
			\label{eq:I3-bound}
		\end{align}
		Thus the interior-node error satisfies
		\begin{equation}
			E_n^\circ\le E_{n+1}+C(\dt^{3/2}+h^2).
			\label{eq:smooth-interior-defect}
		\end{equation}
		
		\medskip
		\noindent\textbf{Boundary nodes.}
		Part~\textup{(II)} of Proposition~\ref{prop:nodewise-consistency} and the
		classical boundary condition give
		\begin{equation}
			\left|S_{n,i}^{\partial}[u_h^n]-u_i^n\right|\le Ch^2,
			\qquad i\in\mathcal I_h^\partial.
			\label{eq:boundary-exact-defect}
		\end{equation}
		Since $U_i^n=S_{n,i}^{\partial}[U^n]$ and
		$|\inf F-\inf G|\le\sup|F-G|$, it follows that
		\begin{equation}
			|e_i^n|\le
			\sup_{b\in B}
			\frac{\displaystyle\sum_{j\in\mathcal I_h^\partial}
				\omega_{ij}^b|e_j^n|
				+\displaystyle\sum_{j\in\mathcal I_h^\circ}
				\omega_{ij}^b|e_j^n|}
			{1+\ell_hk(t_n,x_i,b)}+Ch^2.
			\label{eq:boundary-error-direct}
		\end{equation}
		For small $h$, boundedness of $k$ and~\eqref{eq:boundary-row-margin} give
		\[
		\frac{\sum_{j\in\mathcal I_h^\circ}\omega_{ij}^b}
		{1+\ell_hk(t_n,x_i,b)}\ge\frac\theta2=: \theta_0
		\qquad(i\in\mathcal I_h^\partial,\ b\in B).
		\]
		If $E_n^\partial\le E_n^\circ$, no estimate is needed.  Otherwise, take a
		boundary node at which $E_n^\partial$ is attained.  Since the two normalized
		weight sums in~\eqref{eq:boundary-error-direct} have total at most one and
		the interior sum is at least $\theta_0$, we obtain
		\[
		E_n^\partial
		\le E_n^\partial-\theta_0(E_n^\partial-E_n^\circ)+Ch^2.
		\]
		Therefore
		\begin{equation}
			E_n^\partial\le E_n^\circ+\frac{C}{\theta_0}h^2.
			\label{eq:smooth-boundary-defect}
		\end{equation}
		Combining~\eqref{eq:smooth-interior-defect} and
		\eqref{eq:smooth-boundary-defect} only at this point yields
		\[
		E_n\le E_{n+1}+C(\dt^{3/2}+h^2).
		\]
		Since $E_{N_T}=0$ and $N_T\dt=T$, backward iteration gives
		\[
		E_n\le C(N_T-n)(\dt^{3/2}+h^2)
		\le CT\left(\dt^{1/2}+\frac{h^2}{\dt}\right),
		\]
		which is~\eqref{eq:full-rate}.
	\end{proof}
\section{Numerical experiments}
\label{sec:numerics}

We report four two-dimensional tests for the split scheme~\eqref{eq:scheme}.
The MATLAB implementation precomputes the positive $\mathbb P_1$ rows and all
time-independent branch geometry, retains two time layers, and factorizes each
same-level boundary system once.  The tables contain errors and successive
observed orders; timings and algebraic diagnostics are retained in the
accompanying data files.

\paragraph{Common discretization.}
Experiments~1--3 use the same five nested fitted triangulations of the unit
disk.  Their nominal dyadic parameters are
$h_{\rm nom}=2^{-3},\ldots,2^{-7}$, with respectively
$377$, $1441$, $5633$, $22273$, and $88577$ nodes.  The realized maximum edge
length is $h_{\rm real}=1.243232\,h_{\rm nom}$ on every level; hence the
successive orders are unchanged whether $h_{\rm nom}$ or $h_{\rm real}$ is
used.  Experiment~1 uses $256$ equally spaced unit-circle controls, excluding
the zero control, in both compared implementations.  The main refinement in
Experiment~2 uses the zero control together with $256$ circle directions; its
control sensitivity uses $K=32,64,128,256$ circle directions, always with the
zero control retained.  Experiments~3--4 use the zero control together with
$64$ circle directions.  Both weak Euler branches have probability $1/2$.
Smooth-disk probes are evaluated by the same positive affine-chord
$\mathbb P_1$ map.  Experiment~4 uses its own fitted mixed-boundary meshes.
We set $\ell_h=h_{\rm nom}$.

The angular control error is kept separate from the space--time refinement.
For
\[
 A_K:=\{(\cos(2\pi j/K),\sin(2\pi j/K)):0\le j<K\},
 \qquad H_K(p):=\max_{a\in A_K}a\cdot p,
\]
the nearest angular direction gives
\begin{equation}
 0\le |p|-H_K(p)
 \le \eta_K|p|,
 \qquad \eta_K:=1-\cos(\pi/K).
 \label{eq:angular-control-error}
\end{equation}
For the even values of $K$ used below, adding the zero control does not change
$H_K$, because $A_K$ contains antipodal pairs and hence $H_K(p)\ge0$.
Here $\eta_{64}=1.20454\times10^{-3}$ and
$\eta_{256}=7.52982\times10^{-5}$.  The smooth manufactured solution used in
Experiments~1 and~2 satisfies $\|Du\|_\infty\le3/2$, so the Hamiltonian
residual is bounded by
$1.80682\times10^{-3}$ for $K=64$ and by
$1.12947\times10^{-4}$ for $K=256$.
For the present method, replacing $A$ by $A_K$ adds at most
$\Delta t\,\eta_K\|Du\|_\infty$ to the analytic one-step defect; the recursion
in Theorem~\ref{thm:rate} therefore adds $C\eta_K$ to the full-grid bound.
Thus the tables report the combined space--time and finite-control errors over
the displayed range.  With fixed $K$ an angular-error plateau may eventually
occur; joint asymptotic convergence to the continuous-control problem
requires $K=K(h)\to\infty$.

Every time grid ends at $T=1$: for a prescribed nominal step,
\begin{equation}
 N_T=\operatorname{round}(T/\Delta t_{\rm nom}),
 \qquad \Delta t=T/N_T.
 \label{eq:numerical-time-alignment}
\end{equation}
At $t=0$ we use
\[
 E_\infty=\max_i|U_i-u(0,x_i)|
\]
and, if $\widehat K$ is a curved auxiliary cell with geometric barycentre
$x_{\widehat K}$,
\begin{equation}
 E_1=\sum_{\widehat K}|\widehat K|
 \left|I_h[U](x_{\widehat K})-u(0,x_{\widehat K})\right|.
 \label{eq:barycentric-E1}
\end{equation}
For two consecutive dyadic levels,
\[
 p=\frac{\log(E_{\rm old}/E_{\rm new})}
          {\log(h_{{\rm nom,old}}/h_{{\rm nom,new}})}.
\]
The first displayed level is assigned ``--''.

\subsection{Experiment 1: calibrated \texorpdfstring{$k=0$}{k=0}
comparison on the unit disk}
\label{sec:disk-comparison}

Let
\[
 u(t,x_1,x_2)=(t+0.5)\sin x_1\sin x_2,\qquad
 \sigma(x)=\sqrt2
 \begin{pmatrix}\sin(x_1+x_2)\\ \cos(x_1+x_2)\end{pmatrix}.
\]
With $\mu(a)=-a$, $|a|\le1$, and $r=k=0$, set
\begin{equation}
 f=-u_t-\frac12\sigma^\top D^2u\,\sigma+|Du|.
 \label{eq:experiment2-source}
\end{equation}
On the boundary $g=\gamma\cdot Du$, first with $\gamma=n$ and then with
$\gamma=R_{-\pi/6}n$, where $R_\alpha$ denotes counterclockwise rotation
through $\alpha$.  The reference benchmark is compared after the
corresponding exact time reversal.

For comparison we independently implement the all-node formulas
(25)--(30) of~\cite{CalzolaEtAl2023}; this implementation has no separate
same-level boundary solve.  The present and reference implementations use the
same meshes, time grids, controls, boundary directions, interpolation map, and
error definitions.  We fix $\Delta t=h_{\rm nom}$.  The additional inward
offset $\bar c$ belongs only to the reference formula and is selected using
only errors of that implementation.  On the calibration levels
$h_{\rm nom}=2^{-3},2^{-4},2^{-5}$, we first test
$\bar c=0.050,0.075,\ldots,1.500$.  The score of each candidate is the minimum
of the eight successive calibration orders obtained from two refinements, two
boundary directions, and the two error norms.  This scan selects
$\bar c=0.675$, with minimum and mean scores $0.943$ and $0.979$.
A prespecified refinement over $\bar c=0.650,0.655,\ldots,0.700$ then selects
\begin{equation}
 \bar c=0.690,
 \qquad \min p_{\rm cal}=0.945,
 \qquad \operatorname{mean}p_{\rm cal}=0.979.
 \label{eq:selected-cbar}
\end{equation}
The selected value is frozen before the two finer levels
$h_{\rm nom}=2^{-6}$ and $2^{-7}$ are computed; neither the selection nor its
refinement uses any error of the present method.  The complete scan is included
with the accompanying code.

In Tables~\ref{tab:ex1-normal}--\ref{tab:ex1-oblique}, $P$ and $R$ denote the
present and calibrated reference implementations; each entry is error
(successive order).  The horizontal rule separates calibration from held-out
levels.

\begin{table}[htbp]
 \centering
 \footnotesize
 \setlength{\tabcolsep}{3.1pt}
 \renewcommand{\arraystretch}{0.94}
 \caption{Experiment~1, normal direction, $k=0$, $\Delta t=h_{\rm nom}$,
 and $\bar c=0.690$ for $R$.}
 \label{tab:ex1-normal}
 \begin{tabular}{ccccc}
  \toprule
  $h_{\rm nom}$ & $E_\infty^P(p_\infty^P)$ & $E_\infty^R(p_\infty^R)$
  & $E_1^P(p_1^P)$ & $E_1^R(p_1^R)$\\
  \midrule
  $1/8$   & $6.48845\mathrm e{-2}$ (--)    & $7.67961\mathrm e{-2}$ (--)    & $9.63684\mathrm e{-2}$ (--)    & $1.93530\mathrm e{-1}$ (--)\\
  $1/16$  & $3.74672\mathrm e{-2}$ (0.792) & $3.88512\mathrm e{-2}$ (0.983) & $5.49786\mathrm e{-2}$ (0.810) & $9.90728\mathrm e{-2}$ (0.966)\\
  $1/32$  & $2.02463\mathrm e{-2}$ (0.888) & $1.97056\mathrm e{-2}$ (0.979) & $2.93931\mathrm e{-2}$ (0.903) & $5.14623\mathrm e{-2}$ (0.945)\\
  \midrule
  $1/64$  & $1.05870\mathrm e{-2}$ (0.935) & $1.00446\mathrm e{-2}$ (0.972) & $1.52865\mathrm e{-2}$ (0.943) & $2.55335\mathrm e{-2}$ (1.011)\\
  $1/128$ & $5.46194\mathrm e{-3}$ (0.955) & $5.21256\mathrm e{-3}$ (0.946) & $7.87019\mathrm e{-3}$ (0.958) & $1.28325\mathrm e{-2}$ (0.993)\\
  \bottomrule
 \end{tabular}
\end{table}

\begin{table}[htbp]
 \centering
 \footnotesize
 \setlength{\tabcolsep}{3.1pt}
 \renewcommand{\arraystretch}{0.94}
 \caption{Experiment~1, oblique direction $\gamma=R_{-\pi/6}n$, $k=0$,
 $\Delta t=h_{\rm nom}$, and $\bar c=0.690$ for $R$.}
 \label{tab:ex1-oblique}
 \begin{tabular}{ccccc}
  \toprule
  $h_{\rm nom}$ & $E_\infty^P(p_\infty^P)$ & $E_\infty^R(p_\infty^R)$
  & $E_1^P(p_1^P)$ & $E_1^R(p_1^R)$\\
  \midrule
  $1/8$   & $8.26611\mathrm e{-2}$ (--)    & $7.83111\mathrm e{-2}$ (--)    & $1.09936\mathrm e{-1}$ (--)    & $1.43293\mathrm e{-1}$ (--)\\
  $1/16$  & $4.60009\mathrm e{-2}$ (0.846) & $3.98566\mathrm e{-2}$ (0.974) & $6.03167\mathrm e{-2}$ (0.866) & $7.44555\mathrm e{-2}$ (0.945)\\
  $1/32$  & $2.44043\mathrm e{-2}$ (0.915) & $1.97896\mathrm e{-2}$ (1.010) & $3.19274\mathrm e{-2}$ (0.918) & $3.65230\mathrm e{-2}$ (1.028)\\
  \midrule
  $1/64$  & $1.26330\mathrm e{-2}$ (0.950) & $9.94202\mathrm e{-3}$ (0.993) & $1.65578\mathrm e{-2}$ (0.947) & $1.75893\mathrm e{-2}$ (1.054)\\
  $1/128$ & $6.47359\mathrm e{-3}$ (0.965) & $6.04130\mathrm e{-3}$ (0.719) & $8.50455\mathrm e{-3}$ (0.961) & $8.68951\mathrm e{-3}$ (1.017)\\
  \bottomrule
 \end{tabular}
\end{table}

All displayed errors decrease.  After $\bar c$ is frozen, the orders on
$2^{-5}\to2^{-6}$ are $0.935$--$0.950$ for $P$ and $0.972$--$1.054$ for
$R$; on the fully held-out refinement $2^{-6}\to2^{-7}$ they are
$0.955$--$0.965$ for $P$ and $0.719$--$1.017$ for $R$.  Thus both methods are
competitive over the complete sequence, while the present method gives the
more balanced near-first-order behaviour on the finest refinement without an
additional inward-offset parameter.  These comparisons are empirical; their
relation to Theorem~\ref{thm:rate} is discussed after Experiment~2.

\FloatBarrier

\subsection{Experiment 2: variable Robin coefficient}
\label{sec:variable-robin}

We retain the exact solution and interior equation of Experiment~1 and set
\[
 k(x)=\frac12(1+x_1),\qquad
 g(t,x)=\gamma(x)\cdot Du(t,x)+k(x)u(t,x).
\]
This test exercises both the local-time attenuation and the same-level Robin
closure.

\paragraph{Angular-control sensitivity.}
To check that the fixed angular grid is not controlling the displayed
space--time errors, we freeze $h_{\rm nom}=1/128$ and
$\Delta t=h_{\rm nom}$ and vary only the number $K$ of circle directions.
The zero control is retained, so the total number of controls is $K+1$.
Table~\ref{tab:ex2-control-sensitivity} reports the error and the successive
solution difference
$\delta_K:=\|U_K-U_{2K}\|_\infty$.

\begin{table}[htbp]
 \centering
 \footnotesize
 \setlength{\tabcolsep}{4.2pt}
 \renewcommand{\arraystretch}{0.92}
 \caption{Experiment~2: angular-control sensitivity at
 $h_{\rm nom}=\Delta t=1/128$.  Here $K$ is the number of circle directions;
 the zero control is also included.}
 \label{tab:ex2-control-sensitivity}
 \begin{tabular}{cccccc}
  \toprule
  direction & $K$ & $\eta_K$ & $E_\infty$ & $E_1$ & $\delta_K$\\
  \midrule
  \multirow{4}{*}{normal}
  & 32  & $4.815273\mathrm e{-3}$ & $6.266654\mathrm e{-3}$ & $9.788281\mathrm e{-3}$ & $6.132984\mathrm e{-4}$\\
  & 64  & $1.204544\mathrm e{-3}$ & $5.704361\mathrm e{-3}$ & $8.388478\mathrm e{-3}$ & $1.595100\mathrm e{-4}$\\
  & 128 & $3.011813\mathrm e{-4}$ & $5.557243\mathrm e{-3}$ & $8.035736\mathrm e{-3}$ & $3.916117\mathrm e{-5}$\\
  & 256 & $7.529816\mathrm e{-5}$ & $5.521389\mathrm e{-3}$ & $7.946923\mathrm e{-3}$ & --\\
  \midrule
  \multirow{4}{*}{oblique}
  & 32  & $4.815273\mathrm e{-3}$ & $6.963687\mathrm e{-3}$ & $9.188504\mathrm e{-3}$ & $6.402037\mathrm e{-4}$\\
  & 64  & $1.204544\mathrm e{-3}$ & $6.367182\mathrm e{-3}$ & $7.754582\mathrm e{-3}$ & $1.627866\mathrm e{-4}$\\
  & 128 & $3.011813\mathrm e{-4}$ & $6.216870\mathrm e{-3}$ & $7.405654\mathrm e{-3}$ & $4.044163\mathrm e{-5}$\\
  & 256 & $7.529816\mathrm e{-5}$ & $6.179221\mathrm e{-3}$ & $7.319607\mathrm e{-3}$ & --\\
  \bottomrule
 \end{tabular}
\end{table}

The successive differences have angular orders $1.943$ and $2.026$ in the
normal case and $1.976$ and $2.009$ in the oblique case, consistently with
$\eta_K=O(K^{-2})$.  The zero control is never selected.  The $K=256$
solution is the finest tested discrete-control solution, not the exact
continuous-control solution.

\paragraph{Balanced time-step refinement.}
We next fix $K=256$ and compare
\[
 \Delta t_h=h_{\rm nom},\qquad
 \Delta t_{4/3,\rm nom}=C_{\rm bal} h_{\rm real}^{4/3},
 \qquad C_{\rm bal}=1.496099.
\]
The constant makes the nominal schedules agree on the coarsest mesh.  After
the terminal alignment~\eqref{eq:numerical-time-alignment}, the second
schedule has $N_T=8,20,51,128,323$ and effective successive time exponents
$1.322,1.350,1.328,1.335$.  Table~\ref{tab:ex2-robin} gives both schedules;
each entry is error (successive spatial order).

\begin{table}[htbp]
 \centering
 \scriptsize
 \setlength{\tabcolsep}{2.6pt}
 \renewcommand{\arraystretch}{0.90}
 \caption{Experiment~2: variable Robin coefficient with $K=256$ circle
 directions and the zero control.  Superscripts $h$ and $4/3$ identify
 $\Delta t_h$ and the terminal-aligned $\Delta t_{4/3}$ schedule.}
 \label{tab:ex2-robin}
 \begin{tabular}{cccccc}
  \toprule
  direction & $h_{\rm nom}$ & $E_\infty^h(p_\infty^h)$
  & $E_\infty^{4/3}(p_\infty^{4/3})$
  & $E_1^h(p_1^h)$ & $E_1^{4/3}(p_1^{4/3})$\\
  \midrule
  \multirow{5}{*}{normal}
  & $1/8$   & $7.522077\mathrm e{-2}$ (--)    & $7.522077\mathrm e{-2}$ (--)    & $1.109147\mathrm e{-1}$ (--)    & $1.109147\mathrm e{-1}$ (--)\\
  & $1/16$  & $3.996686\mathrm e{-2}$ (0.912) & $3.236509\mathrm e{-2}$ (1.217) & $5.989028\mathrm e{-2}$ (0.889) & $4.848288\mathrm e{-2}$ (1.194)\\
  & $1/32$  & $2.089982\mathrm e{-2}$ (0.935) & $1.376401\mathrm e{-2}$ (1.234) & $3.089307\mathrm e{-2}$ (0.955) & $2.020889\mathrm e{-2}$ (1.262)\\
  & $1/64$  & $1.078793\mathrm e{-2}$ (0.954) & $5.855086\mathrm e{-3}$ (1.233) & $1.568605\mathrm e{-2}$ (0.978) & $8.576933\mathrm e{-3}$ (1.236)\\
  & $1/128$ & $5.521389\mathrm e{-3}$ (0.966) & $2.469697\mathrm e{-3}$ (1.245) & $7.946923\mathrm e{-3}$ (0.981) & $3.625861\mathrm e{-3}$ (1.242)\\
  \midrule
  \multirow{5}{*}{oblique}
  & $1/8$   & $8.307404\mathrm e{-2}$ (--)    & $8.307404\mathrm e{-2}$ (--)    & $1.093986\mathrm e{-1}$ (--)    & $1.093986\mathrm e{-1}$ (--)\\
  & $1/16$  & $4.521165\mathrm e{-2}$ (0.878) & $3.698557\mathrm e{-2}$ (1.167) & $5.691244\mathrm e{-2}$ (0.943) & $4.625476\mathrm e{-2}$ (1.242)\\
  & $1/32$  & $2.358759\mathrm e{-2}$ (0.939) & $1.543429\mathrm e{-2}$ (1.261) & $2.875803\mathrm e{-2}$ (0.985) & $1.897971\mathrm e{-2}$ (1.285)\\
  & $1/64$  & $1.211431\mathrm e{-2}$ (0.961) & $6.467863\mathrm e{-3}$ (1.255) & $1.449193\mathrm e{-2}$ (0.989) & $7.980794\mathrm e{-3}$ (1.250)\\
  & $1/128$ & $6.179221\mathrm e{-3}$ (0.971) & $2.699367\mathrm e{-3}$ (1.261) & $7.319607\mathrm e{-3}$ (0.985) & $3.375085\mathrm e{-3}$ (1.242)\\
  \bottomrule
 \end{tabular}
\end{table}

\FloatBarrier

All errors decrease.  On the finest mesh the balanced schedule reduces
$E_\infty$ by $55.3\%$ and $56.3\%$ in the normal and oblique cases and
$E_1$ by $54.4\%$ and $53.9\%$, respectively.  Its final observed orders are
$1.245$ and $1.242$ (normal) and $1.261$ and $1.242$ (oblique), compared with
$0.966$--$0.985$ for $\Delta t=h_{\rm nom}$.  All associated boundary margins
are positive and every fallback count is zero.

The continuous-control problems in Experiments~1 and~2 meet the regularity
assumptions of Theorem~\ref{thm:rate}; their implementations additionally
contain the explicitly bounded finite-control defect in
\eqref{eq:angular-control-error}.  For the balanced implementation the bound
has the form $C(h^{2/3}+\eta_{256})$.  The displayed $1.242$--$1.261$ slopes
are therefore pre-asymptotic space--time observations, not a fixed-$K$
asymptotic rate or a sharper global theorem.  A possible explanation for the
strong observed behaviour is the predominantly interior branch geometry.
Only nodes in an $O(\sqrt{\Delta t})$ boundary
layer can generate exiting branches, so most branches on these nearly uniform
disk meshes are unreflected.  On such branches the reflection and local-time
corrections are absent, and the symmetric weak increments cancel the leading
odd Taylor contribution.  For the smooth solutions used here, the analytic
one-step defect is then formally $O(\Delta t^2)$; together with the
$\mathbb P_1$ interpolation error, this gives $O(\Delta t^2+h^2)$ instead of
the uniform $O(\Delta t^{3/2}+h^2)$ estimate required to cover reflected
branches as well.  This interior mechanism is consistent with improved
pre-asymptotic accuracy, but it does not by itself control the maximum norm,
because the near-boundary reflected branches still require uniform control.

\FloatBarrier

\subsection{Experiment 3: degenerate diffusion with a conical viscosity solution}
\label{sec:conical-test}

On the unit disk, let
\[
 u(t,x)=(t+0.5)(1-|x|),\qquad \sigma(x)=\sqrt2\,x,
\]
with $\mu(a)=-a$, $|a|\le1$, $r=0$, and
\[
 f(t,x)=-(1-|x|)+(t+0.5),\qquad
 k(x)=\frac12(1+x_1),\qquad g=\gamma\cdot Du+ku.
\]
The diffusion is rank one away from the origin and vanishes at the origin.
The cone is a viscosity solution.This example lies outside the classical regularity of
Theorem~\ref{thm:rate}.

\begin{table}[htbp]
 \centering
 \footnotesize
 \setlength{\tabcolsep}{5pt}
 \renewcommand{\arraystretch}{0.91}
 \caption{Experiment~3: conical viscosity solution.  Entries are error
 (successive order).}
 \label{tab:ex3-cone}
 \begin{tabular}{ccccc}
  \toprule
  direction & $\Delta t/h_{\rm nom}$ & $h_{\rm nom}$
  & $E_\infty(p_\infty)$ & $E_1(p_1)$\\
  \midrule
  \multirow{10}{*}{normal}
  & \multirow{5}{*}{$1$} & 0.1250000 & $8.6470\mathrm e{-2}$ (--) & $7.1695\mathrm e{-2}$ (--)\\
  && 0.0625000 & $4.2808\mathrm e{-2}$ (1.014) & $3.7889\mathrm e{-2}$ (0.920)\\
  && 0.0312500 & $2.2635\mathrm e{-2}$ (0.919) & $1.9341\mathrm e{-2}$ (0.970)\\
  && 0.0156250 & $1.1752\mathrm e{-2}$ (0.946) & $9.7791\mathrm e{-3}$ (0.984)\\
  && 0.0078125 & $5.8991\mathrm e{-3}$ (0.994) & $4.9266\mathrm e{-3}$ (0.989)\\
  \cmidrule(lr){2-5}
  & \multirow{5}{*}{$1/2$} & 0.1250000 & $6.6301\mathrm e{-2}$ (--) & $6.5815\mathrm e{-2}$ (--)\\
  && 0.0625000 & $3.4015\mathrm e{-2}$ (0.963) & $2.9447\mathrm e{-2}$ (1.160)\\
  && 0.0312500 & $1.8512\mathrm e{-2}$ (0.878) & $1.3567\mathrm e{-2}$ (1.118)\\
  && 0.0156250 & $1.0089\mathrm e{-2}$ (0.876) & $6.4673\mathrm e{-3}$ (1.069)\\
  && 0.0078125 & $5.1133\mathrm e{-3}$ (0.980) & $2.9992\mathrm e{-3}$ (1.109)\\
  \midrule
  \multirow{10}{*}{oblique}
  & \multirow{5}{*}{$1$} & 0.1250000 & $1.0103\mathrm e{-1}$ (--) & $8.5458\mathrm e{-2}$ (--)\\
  && 0.0625000 & $4.9714\mathrm e{-2}$ (1.023) & $4.1696\mathrm e{-2}$ (1.035)\\
  && 0.0312500 & $2.5995\mathrm e{-2}$ (0.935) & $2.0822\mathrm e{-2}$ (1.002)\\
  && 0.0156250 & $1.2951\mathrm e{-2}$ (1.005) & $1.0346\mathrm e{-2}$ (1.009)\\
  && 0.0078125 & $6.6073\mathrm e{-3}$ (0.971) & $5.0938\mathrm e{-3}$ (1.022)\\
  \cmidrule(lr){2-5}
  & \multirow{5}{*}{$1/2$} & 0.1250000 & $7.1741\mathrm e{-2}$ (--) & $7.6609\mathrm e{-2}$ (--)\\
  && 0.0625000 & $3.6819\mathrm e{-2}$ (0.962) & $3.4811\mathrm e{-2}$ (1.138)\\
  && 0.0312500 & $1.9950\mathrm e{-2}$ (0.884) & $1.6228\mathrm e{-2}$ (1.101)\\
  && 0.0156250 & $1.0357\mathrm e{-2}$ (0.946) & $7.7530\mathrm e{-3}$ (1.066)\\
  && 0.0078125 & $5.4404\mathrm e{-3}$ (0.929) & $3.5559\mathrm e{-3}$ (1.125)\\
  \bottomrule
 \end{tabular}
\end{table}

\FloatBarrier

Every error decreases.  The final $E_\infty$ slopes lie in
$0.929$--$0.994$ and the final $E_1$ slopes in $0.989$--$1.125$.
They are empirical slopes for a nonsmooth viscosity solution, not an
application of the classical error theorem.

\FloatBarrier

\subsection{Experiment 4: exploratory mixed Robin--Dirichlet problem}
\label{sec:mixed-test}

Consider
\[
 \Omega=\bigl((-1,1)\times(-0.5,0.5)\bigr)
 \setminus\overline{B_{0.2}(-0.5,0)},\qquad
 u(t,x)=e^{-t}\cos x_1\cos x_2.
\]
The rectangle has the normal Robin condition and the circular hole the exact
Dirichlet value.  We use
\[
 \sigma=(1,1)^\top,\qquad \mu(a)=-a,\qquad r=1,\qquad
 k(x)=\frac12(1+x_1),
\]
and manufacture
\begin{equation}
 f=-u_t-\frac12\sigma^\top D^2u\,\sigma+|Du|+ru,
 \qquad g=n\cdot Du+ku.
 \label{eq:experiment5-pde}
\end{equation}
An interior branch stops at its first hit of the hole; at an outer-face hit it
uses the reflected Robin transform, while outer boundary nodes use the
same-level closure.  Rectangle corners and mixed boundary types violate
Assumption~\ref{ass:H1}, so this is only an exploratory pressure test.

\begin{table}[htbp]
 \centering
 \small
 \setlength{\tabcolsep}{7pt}
 \renewcommand{\arraystretch}{0.96}
 \caption{Experiment~4: exploratory mixed Robin--Dirichlet test, with
 $\Delta t=\ell_h=h_{\rm nom}$.}
 \label{tab:ex4-mixed}
 \begin{tabular}{ccccc}
  \toprule
  $h_{\rm nom}$ & $E_\infty$ & $p_\infty$ & $E_1$ & $p_1$\\
  \midrule
  0.1250000 & $1.9541\mathrm e{-1}$ & -- & $2.2299\mathrm e{-1}$ & --\\
  0.0625000 & $9.6642\mathrm e{-2}$ & 1.016 & $1.1328\mathrm e{-1}$ & 0.977\\
  0.0312500 & $7.8498\mathrm e{-2}$ & 0.300 & $5.8719\mathrm e{-2}$ & 0.948\\
  0.0156250 & $5.8990\mathrm e{-2}$ & 0.412 & $3.1419\mathrm e{-2}$ & 0.902\\
  0.0078125 & $4.6695\mathrm e{-2}$ & 0.337 & $1.7395\mathrm e{-2}$ & 0.853\\
  \bottomrule
 \end{tabular}
\end{table}

Both errors decrease.  The $E_1$ indicator remains close to first order,
whereas the maximum-norm slope slows on the finer grids.  All recorded
fallback counts are zero and all boundary-row margins are positive.  These
diagnostics do not supply a convergence theorem for the corners or mixed
boundary types.

\FloatBarrier
	\section{Conclusion}
	\label{sec7}
	We have separated two mechanisms that approximate two different equations.
	At interior nodes, equal-probability Rademacher branches discretize the
	fixed-coefficient one-step reflected Feynman--Kac identity.
	An exterior branch is mirrored about its oblique projection, its round-trip
	distance represents boundary local time, and the Robin data enter through
	$e^{-kD}$ and $De^{-kD/2}g$.  At boundary nodes, a same-time one-sided
	equation approximates the purely spatial Robin condition over an offset
	$\ell_h$ comparable to $h$.  This separation avoids imposing the volume PDE on the
	boundary.
	
	The boundary closure is generally implicit, including for a linear equation,
	because the $\mathbb P_1$ interpolant may contain unknown boundary values from the
	current time level.  Uniform obliqueness and the standard second-order
	$\mathbb P_1$ interpolation estimate imply a
	uniform positive total weight on interior nodes.  This bound makes the
	controlled boundary map a contraction.  For fixed linear Robin data, the
	resulting matrix is a nonsingular $M$-matrix.
	The complete time-layer update is order preserving.
	
	Under the refinement condition $\Delta t,h\to0$ and
	$h^2/\Delta t\to0$, the scheme satisfies a uniform $L^\infty$ bound and
	converges to the viscosity solution by the half-relaxed-limit argument.  The proof treats a boundary-node
	recovery sequence directly through the discrete Robin equation and an
	interior-node recovery sequence through the reflected expectation identity.
	For a classical HJB solution $u\in C_b^3(\mathcal N_T)$ satisfying the
	uniform boundary-residual condition~\eqref{eq:residual-C11}, the interior error is decomposed into the
	analytic one-step defect, propagation from level $n+1$, and $\mathbb P_1$
	interpolation error.  The boundary error is estimated separately and the two
	node sets are combined only in the full-grid recursion.  Backward iteration gives
	the single estimate
	\[
	\max_{0\le n\le N_T}\max_{i\in\mathcal I_h}
	|U_i^n-u(t_n,x_i)|
	\le C\left(\Delta t^{1/2}+\frac{h^2}{\Delta t}\right).
	\]
	Balancing the two terms gives $\Delta t\asymp h^{4/3}$ and
	$O(h^{2/3})$, while $\Delta t\asymp h$ gives $O(h^{1/2})$.  No algebraic
	rate is claimed for a general continuous viscosity solution.
	
	All four two-dimensional tests were recomputed with the split boundary
	algebra, and every displayed error sequence decreases.  In the $k=0$
	same-grid comparison, a full offset scan followed by local refinement selected
	$\bar c=0.690$ using only the first three levels of the reference formula; the
	parameter was then frozen.  On the pure held-out refinement, the present
	scheme has orders $0.955$--$0.965$ across both boundary directions and both
	norms, whereas the calibrated reference formula has $0.719$--$1.017$.  Over
	the complete sequence the two methods are competitive, while the present
	scheme is more balanced on the finest refinement and requires no additional
	inward-offset coefficient.  For the continuous-control problems in
	Experiments~1 and~2, the regularity hypotheses of
	Theorem~\ref{thm:rate} hold, and the additional finite angular-control defect
	is quantified by~\eqref{eq:angular-control-error}.  In the variable-Robin
	test, doubling the angular resolution reduces the successive solution
	difference at approximately second order, consistently with
	$\eta_K=O(K^{-2})$.  With $K=256$ fixed, the terminal-aligned
	$\Delta t\propto h_{\rm real}^{4/3}$ schedule has final observed orders
	$1.242$--$1.261$ and reduces the finest-grid errors by $53.9\%$--$56.3\%$
	relative to $\Delta t=h_{\rm nom}$.  These pre-asymptotic observations are
	sharper than the conservative $O(h^{2/3})$ balanced bound but do not establish
	a higher global rate.  Since a fixed angular grid eventually produces a
	control-error floor, joint asymptotic convergence to the continuous-control
	problem requires angular refinement as $h\to0$.  The conical test has
	$0.929$--$0.994$ and $0.989$--$1.125$.  For the mixed domain, $E_1$ remains
	close to first order, whereas the maximum-norm slope slows on the finer
	cornered meshes.  Low-regularity rates for the controlled scheme and
	convergence on this nonsmooth mixed domain are not established here.
	
	\section*{Acknowledgments}
	This work was supported by the National Natural Science Foundation of China (No.~12371401).
	
	\section*{Declarations}
	
	\noindent\textbf{Competing interests.}
	The authors declare that they have no competing interests.
	
	\medskip
	\noindent\textbf{Code and data availability.}
	The accompanying self-contained MATLAB package provides one master entry for
	all displayed experiments and tables.  It also contains the full offset scan,
	the angular-control and time-step sensitivities, complete CSV/MAT output, and
	the independent implementation of formulas~(25)--(30)
	of~\cite{CalzolaEtAl2023}; the same-grid rows are labelled by scheme.

\end{document}